\documentclass[12pt]{article}

\usepackage[a4paper, margin=2.5cm]{geometry}

\usepackage{amssymb}
\usepackage{color,graphics}
\usepackage{amsmath}
\usepackage{amsthm}
\usepackage{mathtools}
\usepackage{hyperref}
\usepackage{natbib}
\usepackage{multirow}
\usepackage{multicol}
\usepackage{tikz}
\usepackage{nicematrix}
\usepackage{caption}

\usepackage{mdframed}
\usepackage{lipsum}

\tikzset{
dotstyle/.style={
  inner sep=0pt,
  text width=6mm,
  align=center,
  }
  }

\newtheorem{theorem}{Theorem}[section]
\newtheorem{corollary}[theorem]{Corollary}
\newtheorem{proposition}[theorem]{Proposition}
\newtheorem{remark}[theorem]{Remark}

\newtheorem{definition}[theorem]{Definition}
\newtheorem{example}[theorem]{Example}
\newtheorem{lemma}[theorem]{Lemma}
\newtheorem{hypothesis}{Hypothesis}
\newtheorem*{remark*}{Remark}

\newmdtheoremenv{boxedhypothesis}[hypothesis]{Hypothesis}

\def\RR{{\mathbb{R}}}

\def\NN{{\mathbb{N}}}

\def\GR{\mathcal{GR}}
\def\G{\mathcal{G}}
\def\F{\mathcal{F}}
\def\B{\mathcal{BP}}
\def\X{\mathcal{X}}
\def\D{\mathcal{D}}

\DeclareMathAlphabet{\mathbbold}{U}{bbold}{m}{n}

\def\3{\mathbbold{3}}

\def\r{\mathbf{r}}
\def\c{\mathbf{c}}

\newcommand{\semi}{{\sqsubset}}
\newcommand{\no}{{\not\sqsubset}}

\begin{document}

	\title{Comparing the numbers of subforests and subgraph-degree-tuples}

    \author{Sergei Shteiner\thanks{Independent researcher, Berlin, Germany, \href{mailto:sergei.shteiner@gmail.com}{\texttt{sergei.shteiner@gmail.com}}}
    \and Pavel Shteyner\thanks{Department of Mathematics, Bar-Ilan University, Ramat-Gan, 5290002, Israel, \href{mailto:pavel.shteyner@biu.ac.il}{\texttt{pavel.shteyner@biu.ac.il}}}}
	\maketitle
	
	\medskip

\begin{abstract}
We enumerate the row-column-sums of all square tridiagonal \((0,1)\)-matrices and prove that their count coincides with 
OEIS~A022026 -- %
the number of acyclic subgraphs of the complete $2\times n$ grid graph.
We then extend this correspondence in two independent directions:
\begin{enumerate}
  \item admitting larger sets of matrix entries, and
  \item relaxing the tridiagonal support to broader prescribed sparsity patterns.
\end{enumerate}
The latter leads us to conjecture that, for any bipartite graph~$G$, the number
of its acyclic subgraphs equals the number of degree sequences realized by
subgraphs of~$G$. Moreover, for any non-bipartite graph, the former should be strictly smaller than the latter. We discuss several general approaches and prove these hypotheses for cactus graphs and generalized book graphs.
\end{abstract}

	

\section{Introduction}\label{sec:intro}

The classical Gale-Ryser theorem characterizes the integer vectors that can arise as the row- and column-sums of an \(m\times n\) binary matrix; see~\cite{Brualdi_Ryser_1991}.
In this paper, we focus on the \emph{tridiagonal} case: square \((0,1)\)-matrices whose nonzero entries lie only on the main and the two adjacent diagonals.

As it turns out, the number of distinct row-column-sums of such matrices coincides with OEIS \cite{OEIS} sequence \href{https://oeis.org/A022026}{A022026}.  This sequence also counts the acyclic subgraphs of the complete \(2\times n\) grid graph \( \mathrm{Grid}^2_n \) (Definition~\ref{def:complete_grid_graph}).   Despite the absence of a clear bijection between these objects, we show that this is not a mere coincidence.

On the one hand, we show that introducing colors into $Grid^2_n$ may correspond to expanding the matrix entry set. On the other hand, viewing binary matrices as biadjacency matrices of bipartite graphs translates the aforementioned relation into a purely graph-theoretic one: the number of acyclic subgraphs of $Grid^2_n$ equals the number of distinct vertex-degree tuples realized by its subgraphs.

\textbf{We conjecture that this equality is satisfied if and only if the graph is bipartite.} Moreover, for any non-bipartite graph, the number of acyclic subgraphs should be strictly smaller than the number of vertex-degree tuples. In addition to $Grid^2_n$, we prove these conjectures for several illustrative classes of graphs: complete bipartite graphs (based on the proof of Speyer~\cite{Speyer}), cactus graphs, and generalized book graphs.

The paper is organized as follows. Section~\ref{sec:notation} introduces the necessary notation. In Section~\ref{sec:3n}, we find the number of distinct row-column-sums of binary tridiagonal
matrices and establish the connection to $Grid^2_n$. Section~\ref{sec:k:colors} counts the number of $k$-colored acyclic subgraphs of $Grid^2_n$.
Section~\ref{sec:last} interprets the results in graph language and discusses the main conjectures, Hypotheses \ref{main_hypo} and \ref{<hypo}. In Section \ref{subseq:factorizations}, we propose general reduction techniques. In Sections \ref{subsection:cacti} and \ref{subsection:books}, we prove the conjectures for cactus graphs and generalized book graphs, respectively. Finally, the Appendix contains more technical proofs, including those for the non-bipartite case.

\section{Notation}\label{sec:notation}

We use the following notation throughout the paper:

\begin{itemize}
    \item $M_n$ denotes the set of all real $n \times n$ matrices, while $M_n(0,1)$ denotes the set of all $n \times n$ matrices with entries in $\{0,1\}$.
    
    \item For a matrix $A \in M_n$ we denote its top-left submatrix of order $n-1$ by $A'$.

    \item $e$ denotes the column vector of all ones.

    \item For a matrix $A$, we denote by $\r^A$ the vector of its row sums and by $\c^A$ the vector of its column sums.
    
    \item For a subset $\X\subseteq M_n$, we investigate the set of possible row-column-sums of such matrices. We define:
    $$\GR(\X) = \{(\r^A, \c^A) : A \in \X\}.$$
    This represents the set of all realizable row-column-sum pairs for matrices in $\X$. The notation $\GR$ stands for Gale and Ryser.
    
    \item For a given set of matrices $\X$, we use $\X(\r, \c)$ to denote the subset of all matrices in $\X$ with row sum $\r$ and column sum $\c$.

    \end{itemize}

    We will focus on square matrices with nonzero entries only on specific diagonals:

\begin{itemize}
    
    \item For a set $X\subseteq \RR$ we consider the set of tridiagonal matrices of order $n$ with entries from $X$:
    $$\3_n(X) = \{A \in M_n(X) : a_{ij} \neq 0 \text{ implies } |i-j| \leq 1\}.$$
    For the sake of brevity, we will simply write $\3_n$ instead of $\3_n(\{0, 1\})$ when dealing with $(0,1)$-matrices.

\end{itemize}

All graphs are assumed undirected. We use the following graph notation and terminology:

\begin{itemize}

    \item For a graph $G$ we denote by $V(G)$ the set of its vertices. The set of edges of $G$ is denoted by $E(G)$.

    \item The term {\em subgraph} refers to a not necessarily connected spanning subgraph. That is, a graph $H$  is a {\em subgraph} of a graph $G$ (denoted $H \subseteq G$) if $$\begin{cases}
        V(H) = V(G); \\
        E(H) \subseteq E(G).
    \end{cases}$$
    
    \item Given two graphs $G_1, G_2$, their {\em union} $G_1 \cup G_2$ is the graph satisfying $$\begin{cases}
        V(G_1 \cup G_2) = V(G_1) \cup V(G_2);\\
        E(G_1 \cup G_2) = E(G_1) \cup E(G_2).
    \end{cases}$$ In this paper, we only deal with the unions of edge-disjoint graphs.

    \item Given two graphs $G_1, G_2$, their {\em intersection} $G_1 \cap G_2$ is the graph satisfying \[\begin{cases}
        V(G_1 \cap G_2) = V(G_1) \cap V(G_2);\\
        E(G_1 \cap G_2) = E(G_1) \cap E(G_2).
    \end{cases}\]
   \item A {\em circuit} is a non-empty trail in which the first and last vertices are equal. A {\em cycle} is a circuit with no repeated vertices except the first and the last.

    \end{itemize}

In this paper, we investigate acyclic subgraphs and vertex-degree tuples of subgraphs. We introduce the following notation:

    \begin{itemize}
    \item For a graph $G$ with $n$ vertices, $d(G)$ denotes the ordered $n$-tuple of its vertex degrees. Also $d_G(v)$ denotes the degree of $v \in V(G)$. When there is no danger of confusion, we simply write $d(v)$.
    
    \item Let $\D(G) = \{d(H) : H \subseteq G \}$. In other words, $\D(G)$ is the set of degree sequences of all subgraphs of $G$.

    \item $\F(G)$ is the set of all acyclic,  not necessarily connected subgraphs (i.e. \textit{subforests}) of $G$.

    \item We call $G$ an $\mathsf{FED}$-graph (stands for `F equals D'), if $|\F(G)| = |\D(G)|$. We call $G$ an $\mathsf{FLD}$-graph if $|\F(G)| < |\D(G)|$.
\end{itemize}
    
\section{Row-column-sums of tridiagonal matrices}\label{sec:3n}

        In this section, we compute the cardinality of the set $\GR(\3_n)$, i.e. the number of realizable row-column-sums of binary tridiagonal matrices. We prove that $(|\GR(\3_n)|)_{n \in \NN}$ is the number sequence \href{https://oeis.org/A022026}{OEIS~A022026}. Namely, 
        \begin{equation}\label{formula:GR}\begin{cases}
            |\GR(\3_1)| = 2;\\
            |\GR(\3_2)| = 15;\\
            |\GR(\3_n)| = 8|\GR(\3_{n-1})| - 4|\GR(\3_{n-2})| \text{ for } n > 2.
        \end{cases}\end{equation}
	Remark \ref{rem:n=1} and Lemma \ref{lem:n=2} establish the initial terms.
	
	\begin{remark}\label{rem:n=1} There are exactly two matrices in $\3_1$. Therefore, $\GR(\3_1) = \{((0), (0)), ((1), (1))$\}. In particular, $|\GR(\3_1)| = 2$.
	\end{remark}
	
	\begin{lemma}\label{lem:n=2}
	    $|\GR(\3_2)| = 15$.
	\end{lemma}
{\begin{proof}
    There are $16$ matrices in $\3_2$. Assume that distinct $A, B \in \3_2$ have the same row and column sums. Then $D = A - B$ is a $2\times 2$ $(0,\pm1)$-matrix with zero row and column sums. Every entry of $D$ is nonzero, since otherwise $\r^D = \c^D=0$ forces $D=0$, which contradicts $A \neq B$. Thus, up to a sign, $D = \left[\begin{matrix}
        1 & -1 \\ -1 & 1
    \end{matrix}\right]$. Then $\{A, B\} = \
    \left\{\left[\begin{matrix} 1 & 0 \\ 0 & 1 \end{matrix}\right], \left[\begin{matrix} 0 & 1 \\ 1 & 0 \end{matrix}\right]\right\}.$

    Finally, exactly two matrices in $\3_2$ have the same row and column sums and thus $|\GR(\3_2)| = 15$.
\end{proof}}

    Recall that for a matrix $A \in M_n$ we denote its top-left submatrix of order $n-1$ by $A'$.
    \begin{definition}\label{def:extends}\rm
        We say that a row-column-sum tuple $(\r, \c) \in \GR(\3_{n})$ {\em extends} a row-column-sum tuple $(\r', \c') \in \GR(\3_{n - 1})$, if there exists $A \in \3_{n}(\r, \c)$ such that $A'~\in~\3_{n-1}(\r', \c')$. 

        We call such $A$ an {\em extending matrix}. Note that \begin{equation}\label{eq:extension}
        \begin{cases}
    		\r_i = \r'_i \text{ for } i \leq n - 2; \\
    		\r_{n - 1} = \r'_{n-1} + a_{n-1, n};\\
    		\r_{n} = a_{n, n-1} + a_{n, n}.
    	\end{cases} \qquad\qquad \begin{cases}
    	\c_i = \c'_i \text{ for } i \leq n - 2; \\
    	\c_{n - 1} = \c'_{n - 1} + a_{n, n - 1};\\
    	\c_{n} = a_{n-1, n} + a_{n, n}.
    \end{cases}\end{equation}

    \end{definition}

    Clearly, for $n \geq 2$, every $(\r, \c)\in \GR(\3_n)$ extends at least one row-column-sum tuple from $\GR(\3_{n-1})$. The following lemma shows, when this number is exactly one.
    \begin{lemma}\label{lem:extends_1}
        A row-column-sum tuple $(\r, \c) \in \GR(\3_n)$ with $(\r_n, \c_n) \neq (1, 1)$ extends exactly one row-column-sum tuple from $\GR(\3_{n-1})$.
    \end{lemma}
    \begin{proof}
        Assume that $(\r, \c)$ extends $(\r', \c') \in \GR(\3_{n-1})$ and let $A$ be an extending matrix.
        Without loss of generality assume that $\r_n \neq 1$. There are two possibilities:
        
        \textbf{Case 1.} If $\r_n = 0$, then $a_{n, n - 1} = a_{n, n} = 0$. In this case $a_{n-1, n} = \c_n$.

        \textbf{Case 2.} If $\r_n = 2$, then $a_{n, n - 1} = a_{n, n} = 1$. In this case $a_{n-1, n} = \c_n - 1$.

        Therefore, in both cases we can uniquely determine $a_{n-1, n}, a_{n, n-1}, a_{n, n}$. Then by \eqref{eq:extension} $(\r', \c')$ is also unique.
    \end{proof}

    \begin{remark}\label{rem:extends_2}
        Any matrix $A \in \3_n(\r, \c)$ with $(\r_n, \c_n) = (1, 1)$ is either of the form $\left[\begin{smallmatrix}
	   &  & \vdots\\
	   &  &      1 \\
	\cdots & 1 & 0 
\end{smallmatrix}\right] \text{ or }\left[\begin{smallmatrix}
	   &  & \vdots\\
	   &  &      0 \\
	\cdots & 0 & 1 
\end{smallmatrix}\right]$.
    \end{remark}

    \begin{theorem}\label{thm:3:matrices}
        Let $n > 2$. Then $|\GR(\3_n)| = 8|\GR(\3_{n-1})| - 4|\GR(\3_{n-2})|$.
    \end{theorem}
    \begin{proof}
    We use the inclusion-exclusion approach. 
    
    \textbf{I.} There are $2^3$ ways to extend any $(\r', \c') \in \GR(\3_{n-1})$ to $(\r, \c) \in \GR(\3_{n})$. 
    
    \textbf{II.} Assume that $(\r, \c) \in \GR(\3_{n})$ extends two distinct row-column-sum tuples $(\r^{1}, \c^{1}), (\r^{2}, \c^{2}) \in \GR(\3_{n-1})$. It follows that $\r_n = \c_n = 1$ by Lemma \ref{lem:extends_1}.

    Let $A$, $B$ be matrices that extend $(\r^{1}, \c^{1})$ and $(\r^{2}, \c^{2})$, respectively, to $(\r, \c)$. Then $(a_{n, n-1}, a_{n, n}, a_{n-1, n}) \neq (b_{n, n-1}, b_{n, n}, b_{n-1, n})$ by \eqref{eq:extension}. Hence, due to Remark \ref{rem:extends_2}, we may assume without loss of generality that
    \begin{equation}\label{eq:pair}A = \left[\begin{smallmatrix}
	   &  & \vdots\\
	   &  &      1 \\
	\cdots & 1 & 0 
\end{smallmatrix}\right] \text{ and } B = \left[\begin{smallmatrix}
	   &  & \vdots\\
	   &  &      0 \\
	\cdots & 0 & 1 
\end{smallmatrix}\right].\end{equation} 

It follows that $\r_{n-1}= 
\r^{1}_{n - 1} + 1=\r^{2}_{n - 1}+0$ and $\c_{n-1} = \c^{1}_{n - 1} + 1 = \c^{2}_{n - 1} +0$. Note that $0 \leq \r^{1}_{n - 1}, \r^{2}_{n - 1}, \c^{1}_{n - 1}, \c^{2}_{n - 1} \leq 2$. Then there are four possibilities.

\begin{itemize}
    \item[$(i)$] $(\r^{1}_{n - 1}, \c^{1}_{n - 1}) = (0, 0)$ and $(\r^{2}_{n - 1}, \c^{2}_{n - 1}) = (1, 1)$.
    \item[$(ii)$] $(\r^{1}_{n - 1}, \c^{1}_{n - 1}) = (1, 0)$ and $(\r^{2}_{n - 1}, \c^{2}_{n - 1}) = (2, 1)$.
    \item[$(iii)$] $(\r^{1}_{n - 1}, \c^{1}_{n - 1}) = (0, 1)$ and $(\r^{2}_{n - 1}, \c^{2}_{n - 1}) = (1, 2)$.
    \item[$(iv)$] $(\r^{1}_{n - 1}, \c^{1}_{n - 1}) = (1, 1)$ and $(\r^{2}_{n - 1}, \c^{2}_{n - 1}) = (2, 2)$.
\end{itemize}

Let us consider every possibility.

\begin{itemize}
    \item[$(i)$] By Lemma \ref{lem:extends_1}, \(|\{(\r', \c') \in \GR(\3_{n - 1}) \ : \r'_{n-1} = \c'_{n-1} = 0\}| = |\GR(\3_{n - 2})|.\)
    It follows that the number of collisions of the form $(i)$ doesn't exceed $|\GR(\3_{n - 2})|$. On the other hand, given any $(\r'', \c'')\in \GR(\3_{n-2})$, we can extend it twice with \begin{equation*}A = \left[\begin{smallmatrix}
	   &  & \vdots & \\
	   &  &      0 &  \\
	\cdots & 0 & 0 & 1 \\ 
	   &  & 1 & 0
\end{smallmatrix}\right], \qquad B = \left[\begin{smallmatrix}
	   &  & \vdots & \\
	   &  &      0 &  \\
	\cdots & 0 & 1 & 0 \\ 
	 &  & 0 & 1
\end{smallmatrix}\right]\end{equation*}
to obtain a collision of the form $(i)$. That is, there are exactly $|\GR(\3_{n - 2})|$ collisions of the form $(i)$.
\end{itemize}

 The remaining three cases are almost completely analogous. We provide them fully for the sake of completeness. Notice that in case $(iv)$ we estimate the number of collisions using $(\r^2, \c^2)$ instead of $(\r^1, \c^1)$.
\begin{itemize}
\item[$(ii)$] By Lemma \ref{lem:extends_1}, \(|\{(\r', \c') \in \GR(\3_{n - 1}) \ : \r'_{n-1} = 1, \c'_{n-1} = 0\}| = |\GR(\3_{n - 2})|.\)
    It follows that the number of collisions of the form $(ii)$ doesn't exceed $|\GR(\3_{n - 2})|$. On the other hand, given any $(\r'', \c'')\in \GR(\3_{n-2})$, we can extend it twice with \begin{equation*}A = \left[\begin{smallmatrix}
	   &  & \vdots & \\
	   &  &      0 &  \\
	\cdots & 1 & 0 & 1 \\ 
	   &  & 1 & 0
\end{smallmatrix}\right], \qquad B = \left[\begin{smallmatrix}
	   &  & \vdots & \\
	   &  &      0 &  \\
	\cdots & 1 & 1 & 0 \\ 
	 &  & 0 & 1
\end{smallmatrix}\right]\end{equation*}
to obtain a collision of the form $(ii)$. That is, there are exactly $|\GR(\3_{n - 2})|$ collisions of the form $(ii)$.

\item[$(iii)$] By Lemma \ref{lem:extends_1}, \(|\{(\r', \c') \in \GR(\3_{n - 1}) \ : \r'_{n-1} = 0, \c'_{n-1} = 1\}| = |\GR(\3_{n - 2})|.\)
    It follows that the number of collisions of the form $(iii)$ doesn't exceed $|\GR(\3_{n - 2})|$. On the other hand, given any $(\r'', \c'')\in \GR(\3_{n-2})$, we can extend it twice with \begin{equation*}A = \left[\begin{smallmatrix}
	   &  & \vdots & \\
	   &  &      1 &  \\
	\cdots & 0 & 0 & 1 \\ 
	   &  & 1 & 0
\end{smallmatrix}\right], \qquad B = \left[\begin{smallmatrix}
	   &  & \vdots & \\
	   &  &      1 &  \\
	\cdots & 0 & 1 & 0 \\ 
	 &  & 0 & 1
\end{smallmatrix}\right]\end{equation*}
to obtain a collision of the form $(iii)$. That is, there are exactly $|\GR(\3_{n - 2})|$ collisions of the form $(iii)$.

\item[$(iv)$] By Lemma \ref{lem:extends_1}, \(|\{(\r', \c') \in \GR(\3_{n - 1}) \ : \r'_{n-1} = \c'_{n-1} = 2\}| = |\GR(\3_{n - 2})|.\)
    It follows that the number of collisions of the form $(iv)$ doesn't exceed $|\GR(\3_{n - 2})|$. On the other hand, given any $(\r'', \c'')\in \GR(\3_{n-2})$, we can extend it twice with \begin{equation*}A = \left[\begin{smallmatrix}
	   &  & \vdots & \\
	   &  &      1 &  \\
	\cdots & 1 & 0 & 1 \\ 
	   &  & 1 & 0
\end{smallmatrix}\right], \qquad B = \left[\begin{smallmatrix}
	   &  & \vdots & \\
	   &  &      1 &  \\
	\cdots & 1 & 1 & 0 \\ 
	 &  & 0 & 1
\end{smallmatrix}\right]\end{equation*}
to obtain a collision of the form $(iv)$. That is, there are exactly $|\GR(\3_{n - 2})|$ collisions of the form $(iv)$.
\end{itemize}

\textbf{III.} Note that there are no ``colliding triplets'', since Remark \ref{rem:extends_2} leaves no third option in \eqref{eq:pair}.

Finally, the inclusion-exclusion principle gives $|\GR(\3_{n})| = 8|\GR(\3_{n-1})| - 4|\GR(\3_{n-2})|$.

\end{proof}

\subsection{Subforests of the complete $2\times n$ grid graph}\label{subsec:forests:2n}

\begin{definition}\label{def:complete_grid_graph}\rm
    The {\em complete $2 \times n$ grid graph} $Grid^2_n$ (sometimes also called a ladder graph) is the graph with \[V(Grid^2_n) = \{v^1, \ldots, v^n, v_1, \ldots, v_n\}\] and \[E(Grid^2_n) = \{(v^i, v^{i + 1}) : i < n\} \cup \{(v_i, v_{i + 1}) : i < n\} \cup \{(v^i, v_i) : i \leq n\}.\]
\end{definition}

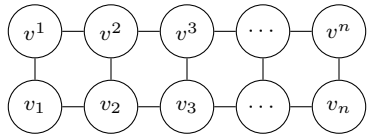
    \begin{center}
   	\begin{tikzpicture}[node distance={10mm}, nodes={draw, minimum size=20pt}, main/.style = {draw, circle}] 
   
   \scriptsize
    
\node[main] (1) [dotstyle] {$v^{1}$}; 
\node[main] (2) [dotstyle] [right of=1] {$v^{2}$};
\node[main] (3) [dotstyle] [right of=2] {$v^{3}$};
\node[main] (4) [dotstyle] [right of=3] {$\cdots$};
\node[main] (5) [dotstyle] [right of=4] {$v^{n}$};

\node[main] (11) [dotstyle] [below of=1]{$v_{1}$}; 
\node[main] (12) [dotstyle] [right of=11] {$v_{2}$};
\node[main] (13) [dotstyle] [right of=12] {$v_3$};
\node[main] (14) [dotstyle] [right of=13] {$\cdots$};
\node[main] (15) [dotstyle] [right of=14] {$v_{n}$};

\draw (1) -- (2);
\draw (2) -- (3);
\draw (3) -- (4);
\draw (4) -- (5);

\draw (11) -- (12);
\draw (12) -- (13);
\draw (13) -- (14);
\draw (14) -- (15);

\draw(1) -- (11);
\draw(2) -- (12);
\draw(3) -- (13);
\draw(4) -- (14);
\draw(5) -- (15);
\end{tikzpicture}

\end{center}

\begin{definition}\rm
    A subgraph of $Grid^2_n$ is called a {\em $2 \times n$ grid graph}.
\end{definition}

    In other words, a $2\times n$ grid graph is a graph whose vertices correspond to the points on the plane with integer coordinates, $x$-coordinates being in the range $1, \ldots, n$, $y$-coordinates being $0$ or $1$. Two vertices can be connected by an edge only if the corresponding points are at distance $1$.  In these terms, $v_i = (i, 0)$ and $v^i = (i, 1)$.

    It turns out that the number $|\GR(\3_n)|$ is also the number of subforests of the complete $2\times n$ grid graph. The explicit formula that coincides with \eqref{formula:GR} was derived in \cite{Desjarlais}. This is also a corollary of our general formula for complete $2\times n$ grid graphs with $k$ colors, see Corollary \ref{cor:first-equality}. This sequence is \href{https://oeis.org/A022026}{OEIS~A022026}.

    Formally, 
    \begin{corollary}\label{cor:GR3n=F_grid}
        $|\GR(\3_n)| = |\F(Grid^2_n)|$.
    \end{corollary}

    Our investigations suggest that this connection is not a coincidence. In the next two sections, we generalize this connection in two different directions.

\section{Subforests of the complete $k$-colored $2\times n$ grid graph}\label{sec:k:colors}

    We may generalize the set of $2\times n$ grid graphs by coloring the edges with $k$ colors. Let $\G(n, k)$ be the set of all $k$-colored graphs on the $2\times n$ grid. Observe that $|\G(n, k)| = (k+1)^{3n - 2}$, since every edge can either be absent or colored with one of the $k$ colors.
    
    We are interested in the number of subforests in this new setting. Naturally, two forests with different colorings are considered different. Let $\F(n, k)$ denote the set of $k$-colored subforests of the complete $2\times n$ grid.

    In this section, we find the formula for $|\F(n, k)|$. In particular, we obtain the result from \cite{Desjarlais} as a particular case for $k = 1$. 

    Case $n = 1$ is trivial:
    \begin{remark}\label{rem:forests:n=1}
        $|\F(1, k)| = k + 1$.
    \end{remark}

    \begin{definition}\rm
        We say that a graph $G \in \G(n, k)$ continues a graph $G' \in \G(n-1,k)$ if $G'$ is an induced subgraph of $G$ with vertices $\{v^1, \ldots, v^{n-1}, v_1, \ldots, v_{n-1}\}$. We call $G'$ {\em the left subgraph of $G$}.
    \end{definition}

    \begin{lemma}
     Let $F \in \F(n-1, k)$.  There are $(k + 1)^3$ ways to continue $F$ to some $G \in \G(n, k)$.
     
     Moreover, if $G \not\in \F(n, k)$, then $\{(v^{n-1}, v^{n}), (v_{n-1}, v_{n}), (v^{n}, v_{n})\} \subseteq E(G)$.
    \end{lemma}
    \begin{proof}
    When continuing the forest $F$, we have a choice whether to add edges $(v^{n-1}, v^{n})$, $(v_{n-1}, v_{n})$ and $(v^{n}, v_{n})$. Thus there are $(k + 1)^3$ ways to continue $F$ to some $G \in \G(n, k)$.
    
    Assume that $G \not\in \F(n, k)$. It follows that there is a cycle in $G$. On the other hand, there are no cycles in $F$. Thus the cycle must contain edges $(v^{n-1}, v^{n}), (v_{n-1}, v_{n})$ and $(v^{n}, v_{n})$.
    \end{proof}
   
   \begin{definition}\label{def:semicyclic} \rm
   Let $F \in \F(n-1, k)$. Let $G \in \G(n, k)$ be a continuation of $F$ with $\{(v^{n-1}, v^n), (v_{n-1}, v_n), (v^{n}, v_n)\} \subseteq E(G)$. We say that $F$ {\em is semicyclic} if $G$ has a cycle.
   \end{definition}
    
    \begin{lemma}\label{lem:semicyclic}
    Let $F \in \F(n-1, k)$. Then $F$ is semicyclic if and only if for some $i \in \{1, \ldots, n-1\}$ we have \begin{equation}\label{eq:semicyclic}
    \{(v^i, v_i)\} \cup \{(v^i, v^{i+1}), \ldots, (v^{n-2}, v^{n-1})\} \cup \{(v_i, v_{i+1}), \ldots, (v_{n-2}, v_{n-1})\} \subseteq E(F).
    \end{equation}
    For $i = n-1$ Condition \eqref{eq:semicyclic} simply means $(v^{n-1}, v_{n-1}) \in E(F)$.
    \end{lemma}
    \begin{proof}
    Let $G\in \G(n, k)$ be a continuation of $F$ with  $\{(v^{n-1}, v^{n}), (v_{n-1}, v_{n}), (v^{n}, v_{n})\} \subseteq E(G)$. Since $F$ is a forest, the only possible cycle in $G$ is

    
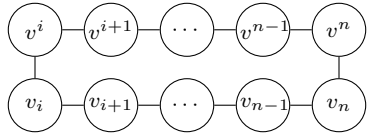
\begin{center}
   	\begin{tikzpicture}[node distance={10mm}, nodes={draw, minimum size=20pt}, main/.style = {draw, circle}] 
   
   \scriptsize
    
\node[main] (1) [dotstyle] {$v^{i}$}; 
\node[main] (2) [dotstyle] [right of=1] {$v^{i+1}$};
\node[main] (3) [dotstyle] [right of=2] {$\cdots$};
\node[main] (4) [dotstyle] [right of=3] {$v^{n-1}$};
\node[main] (5) [dotstyle] [right of=4] {$v^{n}$};

\node[main] (11) [dotstyle] [below of=1]{$v_{i}$}; 
\node[main] (12) [dotstyle] [right of=11] {$v_{i+1}$};
\node[main] (13) [dotstyle] [right of=12] {$\cdots$};
\node[main] (14) [dotstyle] [right of=13] {$v_{n-1}$};
\node[main] (15) [dotstyle] [right of=14] {$v_{n}$};

\draw (1) -- (2);
\draw (2) -- (3);
\draw (3) -- (4);
\draw (4) -- (5);

\draw (11) -- (12);
\draw (12) -- (13);
\draw (13) -- (14);
\draw (14) -- (15);

\draw(1) -- (11);
\draw(5) -- (15);
\end{tikzpicture}

\end{center}

    That is, there exists $i \in \{1, \ldots, n-1\}$ such that 
    \[\{(v^i, v_i)\} \cup \{(v^i, v^{i+1}), \ldots, (v^{n-2}, v^{n-1})\} \cup \{(v_i, v_{i+1}), \ldots, (v_{n-2}, v_{n-1})\} \subseteq E(F).\]
    
    Thus $F$ is semicyclic if and only if Condition \eqref{eq:semicyclic} is satisfied.
  	\end{proof}
    
    The set of the semicyclic forests of $\F(n, k)$ is denoted by $\F^\semi(n, k)$. The set of the forests that are not semicyclic is denoted by $\F^\no(n, k)$.
    
   To simplify further notation we introduce the following sequences for the fixed $k$:
   \[\begin{cases}
   a_{n} = |\F(n, k)|; \\
   a^\semi_{n} = |\F^\semi(n, k)|; \\
   a^\no_{n} = |\F^\no(n, k)|. \\
   \end{cases}\]
   
   Our goal is to find a formula for $a_n$. We begin with a simple observation.
   
   \begin{remark}\label{rem:a1}
   $\begin{cases}
   a_{1} = k + 1; \\
   a^\semi_{1} = k; \\
   a^\no_{1} = 1. \\
   \end{cases}$
   \end{remark}

    Now we can find a recursive formula for $a_n^\no$.
    \begin{lemma}\label{lem:a:no}
    Let $n > 1$. Then $a_{n}^\no = (2k + 1)a_{n - 1} + k^2 a_{n-1}^\no$.
    \end{lemma}
    \begin{proof}
    Let $F \in \F^\no(n, k)$, and let $F' \in \F(n-1,k)$ be its left subgraph.  Then $(v^n, v_n) \not\in E(F)$ by Lemma \ref{lem:semicyclic}. There are two possibilities:

    \begin{enumerate} 
    \item $\{(v^{n-1}, v^n), (v_{n-1}, v_n)\} \not\subseteq E(F)$, i.e. at least one of the edges $(v^{n-1}, v^n), (v_{n-1}, v_n)$ is not in $F$. In this case, $F$ is not semicyclic regardless of $F'$. Thus there are $(k+k+1)a_{n-1}$ non-semicyclic graphs with this property.

    \item $\{(v^{n-1}, v^n), (v_{n-1}, v_n)\} \subseteq E(F)$. In this case, $F$ is not semicyclic if and only if $F'$ is not semicyclic. Thus there are $k^2a^\no_{n-1}$ non-semicyclic graphs with this property.
    \end{enumerate}
    
    Finally,  $a_{n}^\no = (2k + 1)a_{n - 1} + k^2 a_{n-1}^\no$.
     \end{proof}

    Having found the recursive formula for $a_n^\no$, we can find a recursive formula for $a_n^\semi$.
    \begin{lemma}\label{lem:a:semi}
    Let $n > 1$. Then $a_{n}^\semi = ka^\no_{n} +  k^2 a_{n-1}^\semi$.
    \end{lemma}
    \begin{proof}
     Let $F \in \F^\semi(n, k)$ and let $F' \in \F(n-1,k)$ be its left subgraph. There are two possibilities:

     \begin{enumerate}
     \item $(v^n, v_n) \in E(F)$. This condition guarantees that $F$ is semicyclic. Let $\hat{F} \in \F(n, k)$ be such that $E(\hat{F}) = E(F) \setminus \{(v^n, v_n)\}$. Observe that $F$ is a forest if and only if $\hat{F}$ is not semicyclic.
       Thus there are $ka^\no_{n}$ semicyclic forests in $\F(n, k)$ with $(v^n, v_n) \in E(F)$.

\item $(v^n, v_n) \not\in E(F)$. In this case, $F$ is semicyclic if and only if the following two conditions hold:
     \begin{enumerate}
     \item $\{(v^{n-1}, v^n), (v_{n-1}, v_n)\} \subseteq E(F)$.
\item $F'$ is semicyclic.
     \end{enumerate}
     
      It follows that there are $k^2 a^\semi_{n-1}$ semicyclic forests in $\F(n, k)$ with $(v^n, v_n) \not\in E(F)$.
     \end{enumerate}
     Finally, $a_{n}^\semi = ka^\no_{n} +  k^2 a_{n-1}^\semi$.
    \end{proof}

  Now we can explicitly find the second element of the sequence $(a_n)$.
  \begin{corollary}\label{cor:a2}
  $a_2 = 4k^3 + 6k^2 + 4k + 1$.
  \end{corollary}
  \begin{proof}
  By Lemma \ref{lem:a:no} $a^\no_2 = (2k + 1)(k + 1) + k^2 = 3k^2 + 3k + 1$.
  By Lemma \ref{lem:a:semi} $a^\semi_2 = 3k^3 + 3k^2 + k + k^2k = 4k^3 + 3k^2 + k$.
  
  Thus, $a_2 = a^\no_2 + a^\semi_2 = 4k^3 + 6k^2 + 4k + 1$.
  \end{proof}

    Finally, we can find a recursive formula for $a_n$.
  \begin{theorem}\label{thm:an}
  Let $n > 2$. Then $a_{n} = (4k^2 + 3k + 1)a_{n-1} - (k^4 + 2k^3 + k^2)a_{n - 2}$.
  \end{theorem}
  \begin{proof}
  Lemmas \ref{lem:a:no} and \ref{lem:a:semi} state that \(\begin{cases}
  a^\no_{n} = (2k + 1)a_{n-1} + k^2a^\no_{n-1}; \\
  a^\semi_{n} = ka^\no_{n} + k^2a^\semi_{n-1}.
  \end{cases}
  \)
  
\noindent In particular, \begin{align}\label{eq:formula:an}\begin{split}a_{n} &= a^\no_{n} + a^\semi_{n} = (2k + 1)a_{n-1} + k^2{a^\no_{n-1}} + (2k^2 + k)a_{n-1} + k^3{a^\no_{n-1}} + k^2a^\semi_{n - 1}\\ &= (3k^2 + 3k + 1)a_{n-1} + k^3 a^\no_{n-1}.\end{split}\end{align}
    Further, \(a^\no_{n-1} = (2k+1)a_{n-2} + k^2a^\no_{n-2}\). Thus $$a_n =  (3k^2 + 3k + 1)a_{n-1} + (2k^4+k^3)a_{n-2} + k^5a^\no_{n-2}.$$ It follows that $$k^5a^\no_{n-2} = a_n - (3k^2 + 3k + 1)a_{n-1}  - (2k^4 + k^3)a_{n-2}.$$
    Also due to Equation \eqref{eq:formula:an} we conclude $$a_{n-1} =  (3k^2 + 3k + 1)a_{n-2} + k^3 a^\no_{n-2}.$$
 Therefore, $$k^2a_{n-1} = (3k^4 + 3k^3 + k^2)a_{n-2} + a_n - (3k^2 + 3k + 1)a_{n-1}  - (2k^4 + k^3)a_{n-2}.$$
 Finally, $a_n = (4k^2 + 3k + 1)a_{n-1} - (k^4 + 2k^3 + k^2)a_{n - 2}$.
  \end{proof}
  
  As a consequence of Remark \ref{rem:a1}, Corollary \ref{cor:a2} and Theorem \ref{thm:an} we obtain
  \begin{corollary}
  The numbers $|\F(n, k)|$ are defined by $$\begin{cases}
  |\F(1, k)| = a_1 = k + 1; \\
  |\F(2, k)| = a_2 = 4k^3 + 6k^2 + 4k + 1; \\ 
  |\F(n, k)| = a_n = (4k^2 + 3k + 1)a_{n-1} - (k^4 + 2k^3 + k^2)a_{n - 2} \text{ if } n > 2.
  \end{cases}$$
  \end{corollary}
  
  In particular, for uncolored graphs, i.e. for $k=1$, we obtain \begin{corollary}\label{cor:first-equality}
  The sequence $|\F(n, 1)|$ is defined by 
  $$\begin{cases}
 	|\F(1, 1)| = 2; \\
  |\F(2, 1)| = 15; \\ 
  |\F(n, 1)| = 8|\F(n-1, 1)| - 4|\F(n-2, 1)| \text{ if } n > 2.
  \end{cases}$$
  That is, $|\F(n, 1)| = |\GR(\3_{n})|$ for any $n$.
  \end{corollary}

We conjecture that $|\F(n, 2)| = \Big|\GR\Bigl(\3_{n}(\{0, \pm 1\})\Bigr)\Big|$. That is, if we consider $(0, \pm1)$-matrices instead of $(0,1)$-matrices, then we get the same relation between matrices and graphs. So adding the third value corresponds to adding the second color. Moreover, we propose the following general hypothesis. 

\begin{hypothesis}
    Let $k \in \NN$ and $q \in \RR$. Then $$|\F(n, k)| = \Big|\GR\Bigl(\3_{n}(\{q, q+1, \ldots,q + k\})\Bigr)\Big|.$$
\end{hypothesis}

It is easy to see that it is sufficient to verify the hypothesis for $q = 0$. This was done with Python for small $n, k$: $n \leq 5$ and $k \leq 4$.

 \begin{remark*}
    It seems that proving the hypothesis could turn out to be easier rather than directly generalizing the proof of Theorem \ref{thm:3:matrices} for a larger entry set.
 \end{remark*}

\section{The numbers of subforests and degree-tuples of subgraphs}\label{sec:last}

In this section, we first investigate the set of bipartite graphs with $n$ vertices in both parts. Let $\B_n$ denote the set of all such graphs. For illustrative purposes, we refer to the two parts of the vertex set as the {\em top part} and the {\em bottom part}. We will denote the vertices of each part by $1, \ldots, n$. Thus we write $V(G) = (\{1, \ldots, n\}, \{1, \ldots, n\})$.

For every $G \in \B_n$ we consider its biadjacency matrix $A^G$. This is a $(0,1)$-matrix in which rows correspond to the vertices of the top part of $G$, while columns correspond to those of the bottom part. Naturally, $a^G_{i,j} = 1$ if and only if there is an edge between the $i$-th vertex of the top part and the $j$-th vertex of the bottom part. On the other hand, every square $(0,1)$-matrix of order $n$ can be viewed as a biadjacency matrix of some bipartite graph from $\B_n$. This provides a clear bijection between $M_{n}(0,1)$ and $\B_n$.

Moreover, the row sums of $A^G$ correspond to the degrees of the vertices of the top part, while the column sums correspond to the degrees of the bottom part. Therefore $\r^{A^G}$ and $\c^{A^G}$ are the tuples of degrees of the top part and the bottom part, respectively. In particular, this induces a bijection between $\GR(\3_n)$ and $\D(K^3_{n, n})$, where $K^3_{n, n} \in \B_n$ is the {\em complete tridiagonal graph}, see Definition \ref{def:complete_tridiagonal_graph}.

\begin{definition}\rm
    A matrix $A \in \3_n$ is the {\em complete tridiagonal matrix} if $$A =\left[\begin{matrix}
        1 & 1 & 0 & \ldots & \ldots & \ldots \\
        1 & 1 & 1 & 0 & \ldots & \ldots \\ 
        0 & 1 & 1 & 1 & 0 & \ldots & \\
        \ldots & \ldots & \ldots & \ldots  & \ldots & \ldots\\
        \ldots & \ldots & 0 & 1 & 1 & 1\\
        \ldots & \ldots & \ldots & 0 & 1 & 1
    \end{matrix}\right].$$ In other words, $a_{i,j} = 1$ if and only if $|i - j| \leq 1$.
\end{definition}

\begin{definition}\rm\label{def:complete_tridiagonal_graph}
    The {\em complete tridiagonal graph} $K^3_{n, n}$ is the graph from $\B_n$ with a complete tridiagonal biadjacency matrix.
\end{definition}

\begin{example}
    $K^3_{6, 6}$ is the following graph
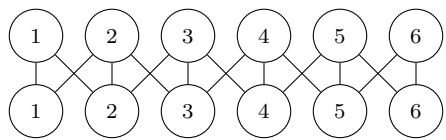
\begin{center}
   	\begin{tikzpicture}[node distance={10mm}, nodes={draw, minimum size=20pt}, main/.style = {draw, circle}] 
   
   \scriptsize
    
\node[main] (1) [dotstyle] {$1$}; 
\node[main] (2) [dotstyle] [right of=1] {$2$};
\node[main] (3) [dotstyle] [right of=2] {$3$};
\node[main] (4) [dotstyle] [right of=3] {$4$};
\node[main] (5) [dotstyle] [right of=4] {$5$};
\node[main] (6) [dotstyle] [right of=5] {$6$};

\node[main] (11) [dotstyle] [below of=1]{$1$}; 
\node[main] (12) [dotstyle] [right of=11] {$2$};
\node[main] (13) [dotstyle] [right of=12] {$3$};
\node[main] (14) [dotstyle] [right of=13] {$4$};
\node[main] (15) [dotstyle] [right of=14] {$5$};
\node[main] (16) [dotstyle] [right of=15] {$6$};

\draw (1) -- (12);
\draw (2) -- (13);
\draw (3) -- (14);
\draw (4) -- (15);
\draw (5) -- (16);

\draw (11) -- (2);
\draw (12) -- (3);
\draw (13) -- (4);
\draw (14) -- (5);
\draw (15) -- (6);

\draw(1) -- (11);
\draw(2) -- (12);
\draw(3) -- (13);
\draw(4) -- (14);
\draw(5) -- (15);
\draw(6) -- (16);
\end{tikzpicture}

\end{center}
\end{example}

\begin{definition}\rm
    A graph from $\B_n$ is {\em tridiagonal} if its biadjacency matrix is tridiagonal; equivalently it is a subgraph of $K^3_{n, n}$.
\end{definition}

Note that there is a clear isomorphism between the complete $2\times n$ grid graph $Grid^2_n$ and the complete tridiagonal graph $K^3_{n, n}$:

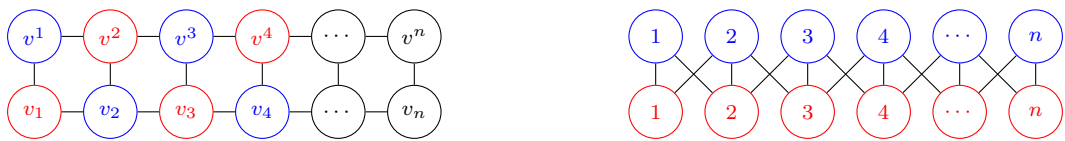
\begin{multicols}{2}

\begin{center}
    
    \begin{tikzpicture}[node distance={10mm}, nodes={draw, minimum size=20pt}, main/.style = {draw, circle}] 
   
   \scriptsize
    
{\color{blue}\node[main] (1) [dotstyle] {$v^1$};
\node[main] (3) [dotstyle] [right of=2] {$v^3$};

\node[main] (12) [dotstyle] [right of=11] {$v_2$};
\node[main] (14) [dotstyle] [right of=13] {$v_4$};
}

\node[main] (5) [dotstyle] [right of=4] {$\cdots$};
\node[main] (16) [dotstyle] [right of=15] {$v_n$};
\node[main] (6) [dotstyle] [right of=5] {$v^n$};
\node[main] (15) [dotstyle] [right of=14] {$\cdots$};

{\color{red}\node[main] (2) [dotstyle] [right of=1] {$v^2$};
\node[main] (4) [dotstyle] [right of=3] {$v^4$};

\node[main] (11) [dotstyle] [below of=1]{$v_1$}; 
\node[main] (13) [dotstyle] [right of=12] {$v_3$};
}

\draw (1) -- (2);
\draw (2) -- (3);
\draw (3) -- (4);
\draw (4) -- (5);
\draw (5) -- (6);

\draw (11) -- (12);
\draw (12) -- (13);
\draw (13) -- (14);
\draw (14) -- (15);
\draw (15) -- (16);

\draw(1) -- (11);
\draw(2) -- (12);
\draw(3) -- (13);
\draw(4) -- (14);
\draw(5) -- (15);
\draw(6) -- (16);
\end{tikzpicture}
\end{center}

\begin{center}
    
    \begin{tikzpicture}[node distance={10mm}, nodes={draw, minimum size=20pt}, main/.style = {draw, circle}] 
   
   \scriptsize
    
{\color{blue}\node[main] (1) [dotstyle] {$1$};
\node[main] (2) [dotstyle] [right of=1] {$2$};
\node[main] (3) [dotstyle] [right of=2] {$3$};
\node[main] (4) [dotstyle] [right of=3] {$4$};
\node[main] (5) [dotstyle] [right of=4] {$\cdots$};
\node[main] (6) [dotstyle] [right of=5] {$n$};}

{\color{red} \node[main] (11) [dotstyle] [below of=1]{$1$}; 
\node[main] (12) [dotstyle] [right of=11] {$2$};
\node[main] (13) [dotstyle] [right of=12] {$3$};
\node[main] (14) [dotstyle] [right of=13] {$4$};
\node[main] (15) [dotstyle] [right of=14] {$\cdots$};
\node[main] (16) [dotstyle] [right of=15] {$n$};}

\draw (1) -- (12);
\draw (2) -- (13);
\draw (3) -- (14);
\draw (4) -- (15);
\draw (5) -- (16);

\draw (11) -- (2);
\draw (12) -- (3);
\draw (13) -- (4);
\draw (14) -- (5);
\draw (15) -- (6);

\draw(1) -- (11);
\draw(2) -- (12);
\draw(3) -- (13);
\draw(4) -- (14);
\draw(5) -- (15);
\draw(6) -- (16);
\end{tikzpicture}
\end{center}

\end{multicols}

That is, for every even $i$, the vertices $v_i$ and $v^i$ are swapped, i.e. \[\begin{pmatrix}
    \textcolor{blue}{v^{2i-1}}\\
    \textcolor{red}{v_{2i-1}}\\
    \textcolor{red}{v^{2i}}\\
    \textcolor{blue}{v_{2i}}
\end{pmatrix}\longleftrightarrow \begin{pmatrix}
    \textcolor{blue}{2i-1}\\
    \textcolor{red}{2i-1}\\
    \textcolor{red}{2i}\\
    \textcolor{blue}{2i}
\end{pmatrix}.\]

The following theorem shows that the number of subforests of a complete tridiagonal graph equals the number of distinct vertex-degree tuples realized by its subgraphs.
\begin{theorem}\label{thm:DK3=FK3}
    $|\D(K^3_{n, n})| = |\F(K^3_{n, n})|$.
\end{theorem}
\begin{proof} We know that $|\GR(\3_n)| = |\D(K^3_{n, n})|$.

    The isomorphism above provides $|\F(K^3_{n, n})| = |\F(Grid^2_n)|$.

    By Corollary \ref{cor:GR3n=F_grid}, $|\F(Grid^2_n)| = |\GR(\3_n)|$.

    Combining these three equalities, we obtain $$|\D(K^3_{n, n})| = |\GR(\3_n)| = |\F(Grid^2_n)| = |\F(K^3_{n, n})|.$$
\end{proof}

Now let's compare the numbers $|\F(G)|$ and $|\D(G)|$ for an arbitrary graph $G$. Recall that $G$ is called an $\mathsf{FED}$-graph if these numbers coincide (stands for `F equals D').

Theorem \ref{thm:DK3=FK3} proves that $K^3_{n, n}$ is an $\mathsf{FED}$-graph. The complete bipartite graph $K_{m, n}$ is also $\mathsf{FED}$:

\begin{proposition}
    The complete bipartite graph $K_{m, n}$ is an $\mathsf{FED}$-graph. That is, \[|\F(K_{m, n})| = |\D(K_{m, n})|.\]
\end{proposition}
\begin{proof}
    This was proved by Speyer; see \cite{Speyer}.
    Note that the original proof deals with matrices. The aforementioned bijection between $M_{n}(0,1)$ and $\B_n$, naturally extended to $m\times n$ matrices and subgraphs of $K_{m, n}$, completes the proof.
\end{proof}

\smallskip
The following theorem provides an important property of two subgraphs with the same vertex-degree tuples. Recall that an Euler graph is a graph with an Euler circuit.

\begin{theorem}\label{thm:sym:euler}
    Let $H_1$, $H_2$ be two distinct subgraphs of a graph $G$. Define a subgraph $H$ of $G$ by \[
E(H) = E(H_1) \,\triangle\, E(H_2) = (E(H_1) \setminus E(H_2)) \cup (E(H_2) \setminus E(H_1)).
\]
The following statements are equivalent:
\begin{enumerate}
    \item $d(H_1) = d(H_2)$;
    \item Every connected component of $H$ is an Euler graph. Moreover, in every component we can choose an Euler circuit in which edges from $E(H_1) \setminus E(H_2)$ and edges from $E(H_2) \setminus E(H_1)$ alternate. In particular, this circuit is of even length.
\end{enumerate}
\end{theorem}
\begin{proof}
    Since $H_1$ and $H_2$ are distinct, the edge set \(E(H)\) is non-empty. Let us call the edges in $E(H_1) \setminus E(H_2)$ red and the edges in $E(H_2) \setminus E(H_1)$ blue.

    First, assume that $d(H_1) = d(H_2)$.     Since $d_{H_1}(v) = d_{H_2}(v)$ for any $v \in V(G)$, we conclude that every vertex $v$ is incident to the same number of red and blue edges. In particular, the degree of each vertex in \(H\) is even. This is exactly the criterion for existence of a color-alternating Euler circuit in every connected component (Kotzig's theorem, \cite[Theorem 7.1]{BangJensen}).

    Now assume that every connected component of $H$ forms an alternating Euler circuit. Then for every $v \in V(G)$ the number of red and blue edges incident with $v$ in $H$ is the same. Thus $d_{H_1}(v) - d_{H_2}(v) = 0$ and therefore $d(H_1) = d(H_2)$.
\end{proof}
As a consequence, we can now easily obtain the following result for acyclic graphs.
\begin{proposition}\label{prop:trees}
    Every acyclic graph is an $\mathsf{FED}$-graph.
\end{proposition}
\begin{proof}
Every subgraph of an acyclic graph is also acyclic. Therefore, it suffices to show that all subgraphs of an acyclic graph have distinct vertex-degree tuples.

Assume the contrary: suppose there exist two distinct subgraphs \( H_1 \) and \( H_2 \) of an acyclic graph \( G \) such that \( d(H_1) = d(H_2) \). Consider another subgraph \( H \subseteq G \) defined by $E(H) = E(H_1) \,\triangle\, E(H_2)$. Then the edge set \( E(H) \) is non-empty, and $H$ contains an Euler circuit by Theorem \ref{thm:sym:euler}, which contradicts the acyclicity of \( G \).
\end{proof}

\begin{remark}
    Acyclic graphs are bipartite.
\end{remark}

Based on the results above, we conjecture
\begin{boxedhypothesis}\label{main_hypo}
    Every bipartite graph is an $\mathsf{FED}$-graph.
\end{boxedhypothesis}

    We verified this with Python for all bipartite graphs with at most $10$ vertices in total.

\medskip

Further, we conjecture that Hypothesis \ref{main_hypo} is actually a characterization of bipartite graphs. Recall that similarly to $\mathsf{FED}$-graphs, a graph $G$ is called an $\mathsf{FLD}$-graph (`F less than D'), if $|\F(G)| < |\D(G)|$.
\begin{boxedhypothesis}\label{<hypo}
    All bipartite graphs are $\mathsf{FED}$-graphs. All non-bipartite graphs are $\mathsf{FLD}$-graphs.
    In other words, $|\F(G)| \leq |\D(G)|$ for any graph $G$ and $G$ is bipartite if and only if $|\F(G)| = |\D(G)|$.
\end{boxedhypothesis}
We verified Hypothesis \ref{<hypo} with Python for all graphs with at most $6$ vertices.

\medskip

Hypothesis \ref{<hypo} is true for trees, as follows from Proposition \ref{prop:trees}. Next we show that it is also true for cycles. Note that a cycle is bipartite if and only if its length is even.

\begin{proposition}\label{prop:cycles}
    Let $C$ be a cycle on $n$ vertices.
    \begin{gather*}
    \text{If  $n$ is even, then } |\F(C)| = |\D(C)| = 2^n - 1;\\
    \text{If  $n$ is odd, then } |\F(C)| = 2^{n} - 1 < |\D(C)| = 2^n.
    \end{gather*}
\end{proposition}
\begin{proof}
    Every subgraph of a cycle is acyclic except for the cycle itself. Therefore $|\F(C)| = 2^n - 1$.

     Now let us find $|\D(C)|$. Denote the edges of $C$ by $e_1 \coloneqq (1, 2), \ldots, e_{n-1} \coloneqq (n-1, n), e_n\coloneqq (n, 1)$. For a fixed edge $e_i$, a subgraph $G$ of $C$ is uniquely determined by $d(G)$ together with the information whether $e_i \in E(G)$. It follows that if $d \in \D(C)$ and $d_i \in \{0, 2\}$ for some $i$, then there is a unique subgraph $G$ of $C$ with $d=d(G)$. This is because $\begin{cases}
         e_i \notin E(G), \text{ if } d_i=0;\\
         e_i \in E(G), \text{ if } d_i=2.
     \end{cases}$

     It remains to consider the case $d=(1, \ldots, 1)$. This can only happen if $n$ is even, since $\sum\limits_{i=1}^n d_i$ is always even. Now, for an even $n$, there are exactly two subgraphs $G_1, G_2$ of $C$ with the vertex-degree tuple $(1, \ldots, 1)$. They are given by $E(G_1) = \{e_1, e_3, \ldots, e_{n-1}\}$ and $E(G_2) = \{e_2, e_4, \ldots, e_{n}\}$.

     Finally, if $n$ is odd, then all subgraphs of $C$ have different vertex-degree tuples. In this case $|\D(C)| = 2^n$. If $n$ is even, then there are exactly two subgraphs $G_1$ and $G_2$ that have the same vertex degrees. Thus $|\D(C)| = 2^n - 1$.
\end{proof}

\bigskip

Proposition \ref{prop:cycles} illustrates the intuition behind Hypotheses \ref{main_hypo} and \ref{<hypo} and shows what differentiates bipartite graphs. Bipartite graphs have only even cycles, while an odd cycle makes the inequality in Hypothesis \ref{<hypo} strict.

\subsection{Factorizations}\label{subseq:factorizations}

In this section, we provide several reduction techniques for graphs with bridges and articulation vertices.

\begin{proposition}\label{prop:articulation:2:F}
    Let graphs $G_1$ and $G_2$ satisfy 
    \[
      V(G_1)\cap V(G_2)=\{v\}, \qquad E(G_1)\cap E(G_2)=\varnothing .
  \] 
  Then
  \[
      |\F(G_1\cup G_2)| \;=\; |\F(G_1)| \, |\F(G_2)|.
  \]
\end{proposition}
\begin{proof}
    Consider a subgraph $H \subseteq G_1 \cup G_2$. Denote $H_1 = G_1 \cap H$ and $H_2 = G_2 \cap H$. Then $H = H_1 \cup H_2$. If either $H_1$ or $H_2$ is not acyclic, then $H$ is also not acyclic. On the other hand, every cycle in $H$ completely lies in $H_1$ or $H_2$, since a cycle cannot pass $v$ more than twice.

    Thus $H$ is acyclic if and only if $H_1, H_2$ are acyclic. It follows that $|\F(G_1\cup G_2)| \;=\; |\F(G_1)| \, |\F(G_2)|$.
\end{proof}

It turns out we cannot as easily perform the same factorization for $|\D(G)|$. We require an additional graph property.

\begin{definition}\rm

Let $G$ be a graph. A vertex $v \in V(G)$ is called {\em degree-determinable} if for any $d_1, d_2 \in \D(G)$ we have
\[d_1(w) = d_2(w) \text{ for all } w \neq v \text{ implies } d_1(v) = d_2(v).\]
\end{definition}

\begin{lemma}\label{lem:determine:bipartite}
    In a bipartite graph every vertex is degree-determinable.
\end{lemma}
\begin{proof}
    Denote the two parts of a bipartite graph $G$ by $\{v_1, \ldots, v_n\}$ and $\{w_1, \ldots, w_m\}$. Since every subgraph of $G$ is also bipartite, $\sum\limits_{i=1}^{n} d(v_i) = \sum\limits_{j=1}^{m}d(w_j)$ for any $d \in \D(G)$. This equality allows us to determine the degree of a vertex, knowing the degrees of the rest of the vertices.
\end{proof}

\begin{remark}
    Lemma \ref{lem:determine:bipartite} doesn't hold for non-bipartite graphs. If $G$ is not bipartite, then there always exist $d_1, d_2 \in \D(G)$ that coincide on all but one vertex. That is, \[\left|\{w \in V(G) : d_1(w) = d_2(w)\}\right| = |V(G)| - 1.\]

    In a non-bipartite graph $G$ there is always a cycle $C = (v_1, v_2, \ldots, v_k, v_1)$ of odd length. Define two subgraphs $G_1, G_2 \subseteq G$ by \begin{gather*}
        E(G_1) = \{(v_1, v_2), (v_3, v_4), \ldots, (v_{k-2}, v_{k-1}), (v_k, v_1)\};\\
        E(G_2) = \{(v_2, v_3), (v_4, v_5), \ldots, (v_{k-1}, v_k)\}.
    \end{gather*}
    Then $d_{G_1}(w) = d_{G_2}(w)=1$ for any $w \in V(G) \setminus\{v_1\}$. But $d_{G_1}(v_1) = 2$, while $d_{G_2}(v_1) = 0$. See an illustration for $k=7$:
    \medskip
    
     
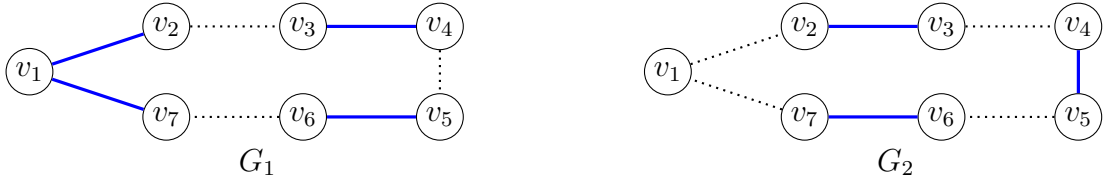
\begin{tikzpicture}[scale=0.6,
        every node/.style={circle,draw,fill=white,inner sep=1.8pt},]   

  \node (v1) at (1,0) {$v_1$};
  \node (v2) at (4,1) {$v_2$};
  \node (v3) at (7,1) {$v_3$};
  \node (v4) at (10,1) {$v_4$};
  \node (v5) at (10,-1) {$v_5$};
  \node (v6) at (7,-1) {$v_6$};
  \node (v7) at (4,-1) {$v_7$};

  \node (w1) at (15,0) {$v_1$};
  \node (w2) at (18,1) {$v_2$};
  \node (w3) at (21,1) {$v_3$};
  \node (w4) at (24,1) {$v_4$};
  \node (w5) at (24,-1) {$v_5$};
  \node (w6) at (21,-1) {$v_6$};
  \node (w7) at (18,-1) {$v_7$};

  \draw[blue, very thick] (v1)--(v2);
  \draw[blue, very thick]  (w2)--(w3);
  \draw[blue, very thick] (v3)--(v4);
  \draw[blue, very thick]  (w4)--(w5);
  \draw[blue, very thick] (v5)--(v6);
  \draw[blue, very thick]  (w6)--(w7);
  \draw[blue, very thick] (v7)--(v1);
  
  \draw[thick, dotted] (w1)--(w2);
  \draw[thick, dotted]  (v2)--(v3);
  \draw[thick, dotted] (w3)--(w4);
  \draw[thick, dotted]  (v4)--(v5);
  \draw[thick, dotted] (w5)--(w6);
  \draw[thick, dotted]  (v6)--(v7);
  \draw[thick, dotted] (w7)--(w1);

\node[draw=none,fill=none] at (6, -2) {$G_1$};
\node[draw=none,fill=none] at (20, -2) {$G_2$};
\node[draw=none,fill=none] at (0, -2) {$ $};
\end{tikzpicture}
\end{remark}

However, a non-bipartite graph may still contain a degree-determinable vertex. For example, a leaf is always degree-determinable:

\begin{lemma}\label{lem:determine:leaf}
    Let $G$ be a graph and let $v \in V(G)$ with $d_G(v) = 1$. Then $v$ is degree-determinable.
\end{lemma}
\begin{proof}
     Consider an arbitrary subgraph $H \subseteq G$. We need to determine $d_H(v)$ by the degrees of the other vertices. Observe that $\sum\limits_{w \in V(G)} d_H(w)$ must be even. Therefore, we can determine the parity of $d_H(v)$. But since $d_H(v) \in \{0, 1\}$, we can uniquely determine $d_H(v)$.
\end{proof}

For $|\D(G)|$ a factorization analogous to Proposition \ref{prop:articulation:2:F} is possible, if $v$ is degree-determinable in $G_1$ or in $G_2$.
\begin{proposition}\label{prop:articulation:2:D}
    Let graphs $G_1$ and $G_2$ satisfy 
    \[
      V(G_1)\cap V(G_2)=\{v\}, \qquad E(G_1)\cap E(G_2)=\varnothing.
  \] 
  If $v$ is degree-determinable in $G_1$, then
  \[
      |\D(G_1\cup G_2)| \;=\; |\D(G_1)| \, |\D(G_2)|.
  \]
\end{proposition}
\begin{proof}
  Write $G:=G_1\cup G_2$.
  Define
  \[
      \Phi : \D(G_1)\times\D(G_2) \longrightarrow \D(G)
      \quad\text{by}\quad
      \Phi(d_{G_1},d_{G_2})(u)=
      \begin{cases}
         d_{G_1}(u), \ u\in V(G_1)\setminus\{v\},\\[4pt]
         d_{G_2}(u), \ u\in V(G_2)\setminus\{v\},\\[4pt]
         d_{G_1}(v)+d_{G_2}(v),\ u=v.
      \end{cases}
  \]
  Let us show that $\Phi$ is bijective.
  
  \noindent\emph{Surjectivity.}
  Take any subgraph $K$ of $G$.
  Setting $\begin{cases} d_1:=d(K\cap G_1)\in\D(G_1),\\
  d_2:=d(K\cap G_2)\in\D(G_2) \end{cases}$ gives 
  $\Phi(d_1,d_2)=d(K)$, so $\Phi$ is surjective.

  \smallskip\noindent
  \emph{Injectivity.}
  Let $d_1, d'_1 \in \D(G_1)$ and $d_2, d'_2 \in \D(G_2)$. Suppose that $\Phi(d_1,d_2)=\Phi(d'_1,d'_2)\eqqcolon d$.
  The two tuples agree on every vertex of $V(G_1)\setminus\{v\}$ and of
  $V(G_2)\setminus\{v\}$, hence
  \[
      \begin{cases} d_1(w) = d'_1(w)=d(w)\text{ for any } w\in V(G_1)\setminus\{v\};\\
      d_2(w) = d'_2(w)=d(w)\text{ for any } w\in V(G_2)\setminus\{v\}.
      \end{cases}
  \]
  But since $v$ is degree-determinable in $G_1$, it must be that $d_1(v) = d'_1(v)$. Hence $d_2(v) = d(v) - d_1(v) = d(v) - d'_1(v) = d'_2(v)$. Finally $(d_1, d_2) = (d'_1, d'_2)$ and therefore $\Phi$ is injective.

  Thus $\Phi$ is bijective and so $|\D(G)| \;=\; |\D(G_1)| \, |\D(G_2)|$.
\end{proof}

The following example shows that Proposition \ref{prop:articulation:2:D} may not hold if $v$ is not degree-determinable in both subgraphs.

\begin{example}
Consider a graph $G$ and its subgraphs $G_1, G_2$, $H, H'$:

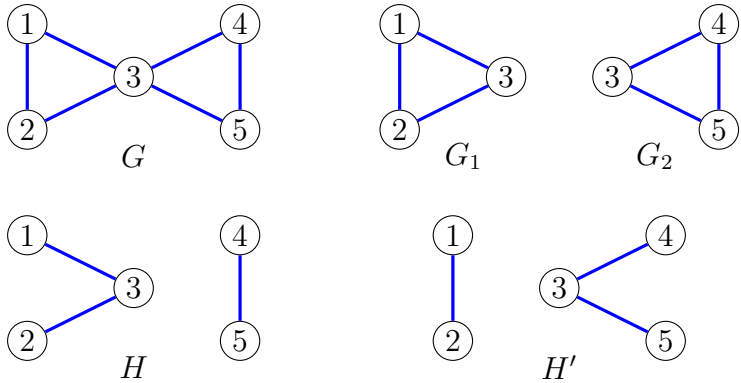
\medskip

\begin{tikzpicture}[scale=0.7,
        every node/.style={circle,draw,fill=white,inner sep=1.8pt},
        edge/.style   ={blue, very thick}]
  \node (G1) at (0, 6) {$1$};   
  \node (G2) at (0, 4) {$2$};   
  \node (G3) at (2, 5) {$3$};   
  \node (G4) at (4, 6) {$4$};   
  \node (G5) at (4, 4) {$5$};   

  \node (G11) at (7, 6) {$1$};   
  \node (G12) at (7, 4) {$2$};   
  \node (G13) at (9, 5) {$3$};   

  \node (G23) at (11, 5) {$3$};   
  \node (G24) at (13, 6) {$4$};   
  \node (G25) at (13, 4) {$5$};   

  \node (H1) at (0, 2) {$1$};   
  \node (H2) at (0, 0) {$2$};   
  \node (H3) at (2, 1) {$3$};   
  \node (H4) at (4, 2) {$4$};   
  \node (H5) at (4, 0) {$5$};   

  \node (H'1) at (8, 2) {$1$};   
  \node (H'2) at (8, 0) {$2$};   
  \node (H'3) at (10, 1) {$3$};   
  \node (H'4) at (12, 2) {$4$};   
  \node (H'5) at (12, 0) {$5$};   

    \draw[edge] (G3) -- (G1) -- (G2) -- (G3) -- (G4) -- (G5) -- (G3);

    \draw[edge] (G13) -- (G11) -- (G12) -- (G13);

    \draw[edge] (G23) -- (G24) -- (G25) -- (G23);

    \draw[edge] (H1) -- (H3);
    \draw[edge] (H2) -- (H3);
    \draw[edge] (H4) -- (H5);

    \draw[edge] (H'1) -- (H'2);
    \draw[edge] (H'3) -- (H'4);
    \draw[edge] (H'3) -- (H'5);

    \node[draw=none,fill=none] at (-2,3.5) {$ $};
    \node[draw=none,fill=none] at (2,3.5) {$G$};
    \node[draw=none,fill=none] at (8.2,3.5) {$G_1$};
    \node[draw=none,fill=none] at (11.8,3.5) {$G_2$};
    \node[draw=none,fill=none] at (2,-0.5) {$H$};
    \node[draw=none,fill=none] at (10,-0.5) {$H'$};
\end{tikzpicture}

Then the vertex $3$ is not degree-determinable in $G_1$, $G_2$. Observe that \begin{align*}
    d_1 \coloneqq d(G_1\cap H) &= (1, 1, 2);\\
    d_2 \coloneqq d(G_2\cap H) &= (0, 1, 1);\\
    d'_1 \coloneqq d(G_1\cap H') &= (1, 1, 0);\\
    d'_2 \coloneqq d(G_2\cap H') &= (2, 1, 1).
\end{align*}
But $\Phi(d_1, d_2) = \Phi(d'_1, d'_2) = (1,1,2,1,1) = d(H) = d(H')$. Therefore the map $\Phi$ from Proposition \ref{prop:articulation:2:D} is not injective, but since it is always surjective, we conclude that $|\D(G)| < |\D(G_1)| |\D(G_2)|$.
\end{example}

The factorizations in Propositions \ref{prop:articulation:2:F} and \ref{prop:articulation:2:D} in some cases allow us to reduce Hypothesis \ref{<hypo} to the respective subgraphs.

\begin{corollary}\label{cor:factorization}
    Assume that graphs $G_1, G_2$ satisfy the conditions of Proposition \ref{prop:articulation:2:D} and let $G = G_1 \cup G_2$. Then
    \[\begin{cases}
    |\F(G)| = |\F(G_1)||\F(G_2)|;\\
    |\D(G)| = |\D(G_1)||\D(G_2)|.
\end{cases}\]
    In particular, if $\begin{cases}
        |\F(G_1)| \leq |\D(G_1)|;\\
        |\F(G_2)| \leq |\D(G_2)|,
    \end{cases}$ then $|\F(G)| \leq |\D(G)|$.
    
    In addition, the inequality is strict if and only if at least one of the inequalities for $G_1, G_2$ is strict. Note also that $G_1 \cup G_2$ is bipartite if and only if both $G_1$ and $G_2$ are bipartite.
\end{corollary}

In particular, Lemma \ref{lem:determine:bipartite} shows us that we can always perform a factorization on a bipartite graph with an articulation point.

The following simple observation allows us to factor out the leaves in a graph when counting $|\F(G)|$ and $|\D(G)|$. Note that if $H$ is a single-edge graph, then it is bipartite and $|\F(H)| = |\D(H)| = 2$.

\begin{remark}\label{rem:factorization:edge}
    Let $G_1, G_2$ be graphs connected by an articulation point. If $G_2$ is a single edge, then $\begin{cases}
    |\F(G_1 \cup G_2)| = 2|\F(G_1)|;\\
    |\D(G_1 \cup G_2)| = 2|\D(G_1)|.
\end{cases}$
\end{remark}

An important class of graphs that allow this factorization are graphs with a bridge, i.e. an edge that does not lie on any cycle.

\begin{proposition}\label{prop:bridge}
Let $G$ be a graph with a bridge $e$. Denote by $H_1$ and $H_2$ the parts of the graph separated by this bridge. Then \[
\begin{cases}
    |\F(G)| = 2|\F(H_1)||\F(H_2)|;\\
    |\D(G)| = 2|\D(H_1)||\D(H_2)|.
\end{cases}
\]
\end{proposition}
\begin{proof}
    Let $H_e$ denote the graph made up of the edge $e$. First, Remark \ref{rem:factorization:edge} implies $\begin{cases}
    |\F(H_e \cup H_1)| = 2|\F(H_1)|;\\
    |\D(H_e \cup H_1)| = 2|\D(H_1)|.
\end{cases}$ Now we apply Corollary \ref{cor:factorization} to the graphs $H_1 \cup H_e$ and $H_2$. This is possible because the graph $H_1 \cup H_e$ satisfies the required conditions by Lemma \ref{lem:determine:leaf}.
\end{proof}

\subsection{Cactus graphs}\label{subsection:cacti}

In this section, using the techniques developed in Section \ref{subseq:factorizations}, we prove Hypothesis \ref{<hypo} for cactus graphs. A connected graph is called a \textit{cactus} if every edge belongs to at most one cycle.

\begin{proposition}\label{prop:cactus}
    If $G$ is a bipartite cactus, then it is an $\mathsf{FED}$-graph. In other words, $|\F(G)| = |\D(G)|$.
\end{proposition}
\begin{proof}
    We prove this by induction on the number of edges of $G$. The statement is trivial for $|E(G)| \leq 2$: if $G$ is acyclic, then $|\F(G)| = |\D(G)|$ by Proposition \ref{prop:trees}.
    
    Assume that the statement is true whenever $|E(G)| \leq m$. Let $G$ be a cactus with $m+1$ edges. Once again, if $G$ is acyclic, then $|\F(G)| = |\D(G)|$. If $G$ is a cycle, then it must be even and the result follows from Proposition \ref{prop:cycles}. Otherwise, the connectivity implies that there is at least one articulation point $v$ in $G$ (take any cycle, it must have a vertex that does not belong to only that cycle). 
    In this case, $G$ can be viewed as a union of two smaller cacti that intersect in $v$.
    Then Corollary \ref{cor:factorization} allows us to reduce the computations to the respective subgraphs, and the induction hypothesis concludes the proof.
\end{proof}

Together with Corollary \ref{cor:nonbipartite:cactus} this proves Hypothesis \ref{<hypo} for cactus graphs:
\begin{theorem}
    All bipartite cacti are $\mathsf{FED}$-graphs. All non-bipartite cacti are $\mathsf{FLD}$-graphs.
    
    In other words, $|\F(G)| \leq |\D(G)|$ for any cactus $G$ and $G$ is bipartite if and only if $|\F(G)| = |\D(G)|$.
\end{theorem}

\subsection{Generalized book graphs}\label{subsection:books}

    In Section \ref{subsection:cacti}, we have proved Hypothesis \ref{<hypo} for cactus graphs, i.e. graphs in which every edge can lie on at most one cycle.  
    In this section, we prove Hypothesis \ref{<hypo} for  generalized book graphs, which is the natural next step.
    
    \begin{definition}\rm
        A {\em generalized book graph}, see \cite{carlson2006books}, is a graph formed by cycles sharing exactly one edge (known as the \textit{spine} or \textit{base} of the book). Each cycle is called a {\em page}.

        For a $k$-page generalized book, let $P_i$ denote the subgraph of the $i$-th page obtained by deleting the base edge. We call the vertices not incident to the base {\em internal}.
    \end{definition}

    In graph theory, a ``book graph'' may refer to several kinds of graphs formed by multiple cycles sharing an edge. For example, it is usually assumed that the cycles are of the same length. So the word ``generalized'' in the name refers to the fact that the lengths of cycles are arbitrary.

    It is relatively easy to find $|\F(B)|$ for a generalized book graph $B$.
    \begin{proposition}\label{prop:F:book}
        Let $B$ be a generalized book graph formed by $k$ cycles of lengths $c_1, \ldots, c_k$. Denote $R \coloneqq \prod\limits_{i=1}^k (2^{c_i - 1} - 1)$. Then \begin{equation}\label{eq:F:book}
            |\F(B)| = 2\cdot R + \sum_{j=1}^{k}\frac{R}{2^{c_j - 1} - 1}.
        \end{equation}
    \end{proposition}
    \begin{proof}
        In order to find the number of acyclic subgraphs of $B$, we count the forests containing the base and not containing the base separately.

        In order to build a forest that contains the base, we can independently choose a subgraph of $P_i$ for every page $i$. The only restriction is that we can't use $P_i$ itself, since together with the base it would form a cycle. Since $P_i$ is an $(c_i - 1)$-edge graph,  we can choose $2^{c_i - 1} - 1$ subgraphs of $P_i$. Thus the number of forests that contain the base is \begin{equation*}
            \prod_{i=1}^k (2^{c_i - 1} - 1) = R.
        \end{equation*}

        For forests that do not contain the base the situation is slightly different. Each of the $R$ forests with the base can be turned into a forest without the base by removing it. But now we have an additional option: a forest without the base can contain all edges of exactly one $P_i$. A subgraph with complete $P_i, P_j$ is not acyclic, since together they form a cycle. Once we choose $P_i$, we choose the subgraphs of the remaining pages as in the previous case.
        Thus the number of forests without the base is \begin{equation*}
            R + \sum_{j = 1}^k \frac{R}{2^{c_j-1} - 1}.
        \end{equation*}
        Together we obtain \eqref{eq:F:book}.
    \end{proof}
\smallskip
    For $|\D(B)|$ the situation is more complex. In particular, we need to take into account the numbers of odd and even pages in $B$. We start with a bipartite case, i.e. when all pages are even cycles. We first make the following technical observation.

    \begin{lemma}\label{lem:deg:0,2}
        Let $P$ be a path $(l, v_1, \ldots, v_n, r)$, where $n$ is even. Let $d \in \D(P)$. Assume that $d(v_i)$ is known for all $i = 1, \ldots, n$. There are two cases:
        \begin{enumerate}
            \item If $d(v_i) \in \{0, 2\}$ for some $i$, then we can uniquely determine $d(l)$ and $d(r)$.

            \item If $d(v_i) = 1$ for every $i$, then there are exactly two possibilities: either $d(l) = d(r) = 0$, or $d(l) = d(r) = 1$.
        \end{enumerate}
         Moreover, in $\D(P)$ there are $2^{n+1} - 2$ tuples of the first type.
    \end{lemma}
    \begin{proof}
        Knowing degrees of all internal vertices ($v_1, \ldots, v_n$) of a subgraph $H \subseteq P$ and whether a certain edge belongs to $E(H)$ is enough to uniquely determine $H$. This is the case if $d(v_i) \neq 1$ for some $i$: in this case we can determine whether the edges incident to $v_i$ belong to a subgraph, and thus determine the entire subgraph uniquely.

        There are exactly two subgraphs with $d(v_i) = 1$ for every $i = 1, \ldots, n$. The first one is obtained by taking every second edge starting from $(l, v_1)$. The second one is obtained by taking every second edge starting from $(v_1, v_2)$. Since $n$ is even, we obtain that in the former case $d(l) = d(r) = 1$, while in the latter case $d(l) = d(r) = 0$, see the illustration below:

    \medskip
        
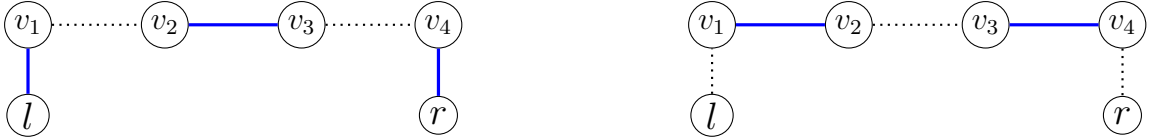
\begin{tikzpicture}[scale=0.6,
        every node/.style={circle,draw,fill=white,inner sep=1.8pt},
        edge/.style   ={very thick}]   


  \node (ll) at (0,0) {\large $l$};
  \node (l1) at (0,2) {$v_1$};
  \node (l2) at (3,2) {$v_2$};
  \node (l3) at (6,2) {$v_3$};
  \node (l4) at (9,2) {$v_4$};
  \node (lr) at (9,0) {\large $r$};
  
  \node (rl) at (15,0) {\large $l$};
  \node (r1) at (15,2) {$v_1$};
  \node (r2) at (18,2) {$v_2$};
  \node (r3) at (21,2) {$v_3$};
  \node (r4) at (24,2) {$v_4$};
  \node (rr) at (24,0) {\large $r$};

  \draw[blue, very thick] (ll)--(l1);
  \draw[blue, very thick] (l2)--(l3);
  \draw[blue, very thick] (l4)--(lr);
  
  \draw[blue, very thick] (r1)--(r2);
  \draw[blue, very thick] (r3)--(r4);

  \draw[dotted, thick] (rl)--(r1);
  \draw[dotted, thick] (r2)--(r3);
  \draw[dotted, thick] (r4)--(rr);
  
  \draw[dotted, thick] (l1)--(l2);
  \draw[dotted, thick] (l3)--(l4);

\end{tikzpicture}

        We have shown that vertex-degree tuples of all but two subgraphs of $P$ satisfy the first condition. Note that since $P$ is acyclic, all its subgraphs have distinct vertex-degree tuples by Theorem \ref{thm:sym:euler}. Thus in $\D(P)$ there are $2^{n+1} - 2$ tuples satisfying the first condition.
    \end{proof}

    \begin{proposition}\label{prop:D:bipartite:book}
        Let $B$ be a bipartite generalized book graph formed by $k$ cycles of lengths $c_1, \ldots, c_k$. For every subset $S \subseteq \{1, \ldots, k\}$ denote \begin{equation*}T_S \coloneqq \prod\limits_{i \notin S} (2^{c_i - 1} - 2).\end{equation*}
        Then \begin{equation}\label{eq:D:bipartite:book}
            |\D(B)| = \sum_{S \subseteq \{1, \ldots, k\}} (|S| + 2) T_S.
        \end{equation}
    \end{proposition}
    \begin{proof}
    Since $B$ is bipartite, every $c_i$ must be even.  Denote the base of $B$ by $e = (l, r)$.
    
    \textbf{Step 1.} Observe that $B$ is the union of $P_1, \ldots, P_k$ and the base. Given a subgraph $H$ of $B$, we can uniquely determine $d\coloneqq d(H)$ by $d_1\coloneqq d(P_1\cap H), \ldots, d_k \coloneqq d(P_k \cap H)$ and whether $e 
    \in E(H)$. Namely, \begin{equation*} d(v) = \begin{cases}
        d_i(v), \ v \in V(P_i) \setminus \{l, r\};\\
        \delta + d_1(l) + \ldots + d_k(l), \ v = l;\\
        \delta + d_1(r) + \ldots + d_k(r), \ v = r,
    \end{cases} \text{ where} \quad \delta = \begin{cases}
        1, \ e \in E(H);\\
        0, \ e \not\in E(H).
    \end{cases}\end{equation*}

    \textbf{Step 2.} Consider $d_i \in \D(P_i)$. There are two possibilities. If $d_{i}(v) \in \{0, 2\}$ for some internal vertex $v$ of $P_i$, then due to Lemma \ref{lem:deg:0,2}, we can uniquely determine $d_{i}(l)$ and $d_{i}(r)$ by the degrees of all the internal vertices. If $d_{i}(v) = 1$ for all internal vertices of $P_i$, then there are two possibilities: $d_i(l) = d_i(r) = 0$ or $1$.

    \textbf{Step 3.} Assuming that we know the degrees of all internal vertices, how many elements in $\D(B)$ match them? That is, given the degrees of the internal vertices, how many options do we have for the degrees of $l, r$? It turns out, the answer depends on how many pages have all internal vertices of degree $1$. We split $\D(B)$ by the number of such pages. Consider arbitrary $S \subseteq \{1, \ldots, k\}$. Let us find the number of elements in $\D(B)$ such that $S$ is the set of pages with all internal vertices of degree $1$.

    \textbf{Step 4.} First, we can choose the internal-vertex degrees of the remaining $k - |S|$ pages in any way, except for all ones. By Lemma \ref{lem:deg:0,2} there are $2^{c_i - 1} - 2$ ways to choose the internal degrees of a page $P_i$ for $i \not \in S$. Thus there are $T_S$ ways to assign internal degrees. Now for each internal-vertex-degree configuration, how many elements in $\D(B)$ match it? Since we know the degrees of all internal vertices, the question becomes how many options do we have for the degrees of $l, r$? Recall that \begin{equation}\label{eq:D:l,r}
        \begin{cases}
        d(l) = \delta + \sum\limits_{i \in S} d_i(l) + \sum\limits_{j \notin S} d_j(l);\\
        d(r) = \delta + \sum\limits_{i \in S} d_i(r) + \sum\limits_{j \notin S} d_j(r).
        \end{cases}
    \end{equation}
    In \eqref{eq:D:l,r} we can choose $\delta \in \{0, 1\}$ and we can choose $d_i(l) = d_i(r) \in \{0, 1\}$ for any $i \in S$, while $\sum\limits_{j \notin S} d_j(l)$ and $\sum\limits_{j \notin S} d_j(r)$ are already determined. It follows that $\delta + \sum\limits_{i \in S} d_i(l) = \delta + \sum\limits_{i \in S} d_i(r) \in \{0, \ldots, |S| + 1\}$, which gives $|S| + 2$ options for $d(l), d(r)$, see Example \ref{example:Step4} for illustration.

    Finally, once we take into account every possible subset $S$ of $\{1, \ldots, k\}$, we obtain \begin{equation*}
            |\D(B)| = \sum_{S \subseteq \{1, \ldots, k\}} (|S| + 2) T_S.
        \end{equation*}
    \end{proof}
The following example illustrates Step 4 of the proof.
\begin{example}\label{example:Step4}
Consider the case $|S| = 2$. Then there are four options for the degrees of $l$ and $r$, more precisely, for $\delta + \sum\limits_{i \in S} d_i(l)$. On the picture below there are $4$ two-page books with all internal vertices of degree $1$. The degrees of $l, r$ can vary from $0$ to $3$.

\bigskip

    
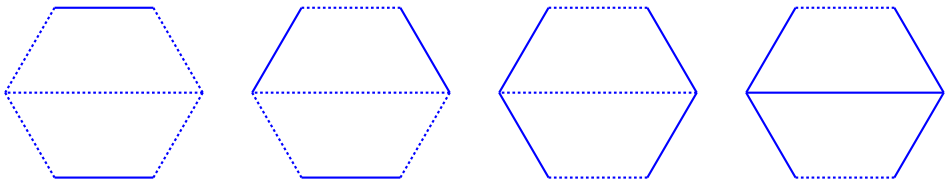
\begin{tikzpicture}[scale=0.65,
        thick,                                
        blue,                                 
        dotted/.style={dash pattern=on 1pt off 1pt}, 
        declare function={R=2;}               
      ]
  \newcommand*\Hex[8]{%
    \begin{scope}[xshift=#1cm]               
      \foreach \i/\sty in {0/#2,1/#3,2/#4,3/#5,4/#6,5/#7}{%
        \draw[\sty] (\i*60:R) -- ({(\i+1)*60}:R);
      }
      \draw[#8] (  0:R) -- (180:R);
    \end{scope}
  }

  \Hex{15}{solid}{dotted}{solid}{solid}{dotted}{solid}{solid} 
  \Hex{10}{solid}{dotted}{solid}{solid}{dotted}{solid}{dotted} 
  \Hex{5}{solid}{dotted}{solid}{dotted}{solid}{dotted}{dotted} 
  \Hex{0}{dotted}{solid}{dotted}{dotted}{solid}{dotted}{dotted} 

\node[draw=none,fill=none] at (-3, 0) {$ $};
  
\end{tikzpicture}

\bigskip

Note that these are all the possibilities for the degrees of $l, r$, but not all possible subgraphs. Namely, there are two more subgraphs that realize $d(l)=d(r) = 1$ and two more subgraphs that realize $d(l)=d(r) = 2$. In total we get $8=2^3$, which is the number of subgraphs of a two-page book with all internal vertices of degree $1$.
\end{example}

As it turns out, the expressions \eqref{eq:F:book} and \eqref{eq:D:bipartite:book} are equal. This is shown in Corollary \ref{cor:equality:books}. Together with Lemma \ref{lem:nonbipartite:book}, this proves Hypothesis \ref{<hypo} for generalized books: 

\begin{theorem}

    All bipartite generalized books are $\mathsf{FED}$-graphs. All non-bipartite generalized books are $\mathsf{FLD}$-graphs.
    
    In other words, $|\F(G)| \leq |\D(G)|$ for any generalized book $G$ and $G$ is bipartite if and only if $|\F(G)| = |\D(G)|$.
\end{theorem}

Moreover, Hypothesis \ref{<hypo} holds for any subgraph of a generalized book graph, see Remark \ref{rem:subbook}.

\section*{Acknowledgments} The authors are grateful to Professor Geir Dahl and Professor Alexander Guterman for valuable discussions on graph theory and combinatorial matrix theory. 

\appendix
\renewcommand{\thesection}{\Alph{section}}
\renewcommand{\thesubsection}{\Alph{section}.\arabic{subsection}}
\section{Appendix}\label{appendix}

\subsection{Non-bipartite cacti}

When $G_1$ is an odd cycle, we can estimate $|\D(G_1 \cup G_2)|$ analogously to Proposition \ref{prop:articulation:2:D}.

\begin{lemma}\label{lem:articulation:odd_cycle}
    Let $C$ be an odd cycle and let graph $G_2$ satisfy 
    \[
      V(C)\cap V(G_2)=\{v\}, \qquad E(C)\cap E(G_2)=\varnothing .
  \] 
  Then
  \[
      |\D(C\cup G_2)| \; > \; |\F(C)| \, |\D(G_2)|.
  \]
\end{lemma}
\begin{proof}

    Denote $C = (v, u_1, \ldots, u_{n}, v)$, where $n$ is even, since $C$ is odd. Let $d \in \D(C)$. Assume that $d(u_i)$ is known for all $i = 1, \ldots, n$.  If $d(u_i) \in \{0, 2\}$ for some $i$, then we can uniquely determine $d(v)$.
    
    Otherwise $d(u_i) = 1$ for all $i$, and there are exactly two subgraphs with this condition: take every second edge of $C$ starting from $(v,u_1)$ or $(u_1,u_2)$. Thus $d(v) \in \{0, 2\}$, see an illustration for $n = 6$:

        \medskip
    
     
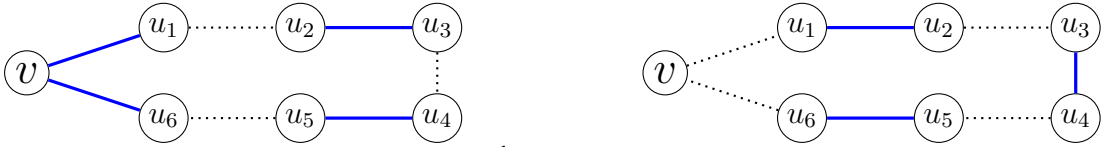
\begin{tikzpicture}[scale=0.6,
        every node/.style={circle,draw,fill=white,inner sep=1.8pt},]   

  \node (v1) at (1,0) {\Large $v$};
  \node (v2) at (4,1) {$u_1$};
  \node (v3) at (7,1) {$u_2$};
  \node (v4) at (10,1) {$u_3$};
  \node (v5) at (10,-1) {$u_4$};
  \node (v6) at (7,-1) {$u_5$};
  \node (v7) at (4,-1) {$u_6$};

  \node (w1) at (15,0) {\Large $v$};
  \node (w2) at (18,1) {$u_1$};
  \node (w3) at (21,1) {$u_2$};
  \node (w4) at (24,1) {$u_3$};
  \node (w5) at (24,-1) {$u_4$};
  \node (w6) at (21,-1) {$u_5$};
  \node (w7) at (18,-1) {$u_6$};

  \draw[blue, very thick] (v1)--(v2);
  \draw[blue, very thick]  (w2)--(w3);
  \draw[blue, very thick] (v3)--(v4);
  \draw[blue, very thick]  (w4)--(w5);
  \draw[blue, very thick] (v5)--(v6);
  \draw[blue, very thick]  (w6)--(w7);
  \draw[blue, very thick] (v7)--(v1);
  
  \draw[thick, dotted] (w1)--(w2);
  \draw[thick, dotted]  (v2)--(v3);
  \draw[thick, dotted] (w3)--(w4);
  \draw[thick, dotted]  (v4)--(v5);
  \draw[thick, dotted] (w5)--(w6);
  \draw[thick, dotted]  (v6)--(v7);
  \draw[thick, dotted] (w7)--(w1);

\end{tikzpicture}

\noindent  Let $d'_C \in \D(C)$ be defined by $d_C'(w) = \begin{cases}
        0, \ w=v;\\
        1, \ w\neq v.
    \end{cases}$ Denote $\D'(C) = \D(C) \setminus \{d'_C\}$. Then $|\D'(C)| = 2^{n+1} - 1 = |\F(C)|$ by Proposition \ref{prop:cycles}.

    Now, consider a map $\Phi$ from Proposition \ref{prop:articulation:2:D}. Its restriction to $\D'(C)\times \D(G_2)$ is injective and therefore $|\D(C\cup G_2)| \geq |\F(C)||\D(G_2)|$. To prove that the inequality is strict, consider a subgraph $H$ of $G$ obtained by a union of the empty subgraph of $G_2$ and the subgraph of $C$ that yields $d'_C$. Then $d_H(v) = 0$, while $d_H(u_i) = 1$ for $i = 1, \ldots, n$. Then $d_H \notin \operatorname{Im}\left(\Phi_{\vert\D'(C)\times \D(G_2)}\right)$, by the construction of $\D'(C)$.
\end{proof}

\begin{corollary}\label{cor:nonbipartite:cactus}
    If $G$ is a non-bipartite cactus, then $|\F(G)| < |\D(G)|$.
\end{corollary}
\begin{proof}
    We prove this by induction on the number of edges of $G$. The smallest non-bipartite graph is the triangle. It satisfies the inequality by Proposition \ref{prop:cycles}. 
    
    Assume that the statement is true whenever $|E(G)| \leq m$. A cactus always has a leaf or a leaf cycle (a cycle that meets the rest of the graph in exactly one vertex, i.e., an articulation point). If $G$ has a leaf, then we can discard it by Remark \ref{rem:factorization:edge} and apply the induction hypothesis.
    
    Let $G_1$ be a leaf cycle, and let $G_2$ denote the rest of the graph. If $G_2$ is empty, then $G=G_1$ is an odd cycle and thus satisfies the inequality by Proposition \ref{prop:cycles}. Assume that $G_2$ is non-empty.
    
    If $G_1$ is an even cycle, then $G_2$ is non-bipartite. Then by induction hypothesis $|\F(G_2)| < |\D(G_2)|$. Since $G_1$ is bipartite, we can apply Proposition \ref{prop:articulation:2:F} to $\F(G)$, Proposition \ref{prop:articulation:2:D} to $\D(G)$ and conclude that $|\F(G)| < |\D(G)|$.

    Finally, assume that $G_1$ is an odd cycle. If $G_2$ is bipartite, then $|\F(G_2)| = |\D(G_2)|$. If $G_2$ is not bipartite, then $|\F(G_2)| < |\D(G_2)|$ by the induction hypothesis. Then Lemma \ref{lem:articulation:odd_cycle} implies that $|\D(G)| > |\F(G_1)||\D(G_2)| \geq |\F(G_1)||\F(G_2)| = |\F(G)|$, where the last equality follows from Proposition \ref{prop:articulation:2:F}.
\end{proof}

\subsection{Books and non-bipartite books}
\begin{lemma}\label{lem:F_1=D_1}
    Define functions $F_1, D_1 : \NN^k \longrightarrow \RR$ by \begin{gather*}
        F_1(a_1, \ldots, a_k) = \prod\limits_{i=1}^k a_i;\\
        D_1(a_1, \ldots, a_k) = \sum_{S \subseteq \{1, \ldots, k\}} \prod_{i \notin S} (a_i - 1).
    \end{gather*}
    Then $F_1 = D_1$.
\end{lemma}
\begin{proof}
Consider finite sets $A_i$ of size $a_i$, and let $A = A_1 \times \cdots \times A_k$ be their Cartesian product. Then $F_1$ equals the cardinality of $A$, i.e., the number of $k$-tuples $(x_1, \ldots, x_k)$ with $x_i \in A_i$.

Now in every $A_i$ fix an arbitrary element $x_i^* \in A_i$.

For each subset $S \subseteq \{1, \ldots, k\}$, define a subset $T_S \subseteq A$ consisting of all tuples $(x_1, \ldots, x_k)$ such that $x_i = x_i^*$ for all $i \in S$, and $x_i \in A_i \setminus \{x_i^*\}$ for all $i \notin S$. Then
\[
|T_S| = \prod_{i \notin S} (a_i - 1).
\]
Since the sets $T_S$ are disjoint and their union is the entire set $A$, we obtain
\[
|A| = \sum_{S \subseteq \{1, \ldots, k\}} |T_S| = \sum_{S \subseteq \{1, \ldots, k\}} \prod_{i \notin S} (a_i - 1),
\]
which is exactly the definition of $D_1(a_1, \ldots, a_k)$. Therefore, $F_1 = D_1$.
\end{proof}

\begin{lemma}\label{lem:F_2=D_2}
    Define functions $F_2, D_2 : \NN^k \longrightarrow \RR$ by \begin{gather*}
        F_2(a_1, \ldots, a_k) = \sum\limits_{j=1}^k\frac{\prod\limits_{i=1}^k a_i}{a_j};\\
        D_2(a_1, \ldots, a_k) = \sum\limits_{S \subseteq \{1, \ldots, k\}}|S| \cdot \prod\limits_{i\not\in S}(a_i - 1).
    \end{gather*}
    Then $F_2 = D_2$.
\end{lemma}
\begin{proof}
For each index $j\in\{1,\dots,k\}$ write
\[
F_1\bigl(a_1,\dots,a_{j-1},a_{j+1},\dots,a_k\bigr)
  =\prod_{\substack{i=1\\ i\neq j}}^{k} a_i
  =\frac{\prod_{i=1}^{k} a_i}{a_j}.
\]
Hence
\[
F_2(a_1,\dots,a_k)=\sum_{j=1}^{k}
      F_1\bigl(a_1,\dots,a_{j-1},a_{j+1},\dots,a_k\bigr).
\]

By Lemma \ref{lem:F_1=D_1}, each summand equals
\[
D_1\bigl(a_1,\dots,a_{j-1},a_{j+1},\dots,a_k\bigr)
      =\sum_{T\subseteq\{1,\dots,k\}\setminus\{j\}}
        \prod_{i\notin T\cup\{j\}}(a_i-1).
\]
Summing these identities over $j$ gives
\[
F_2(a_1,\dots,a_k)=\sum_{j=1}^{k}\;
   \sum_{T\subseteq\{1,\dots,k\}\setminus\{j\}}
       \prod_{i\notin T\cup\{j\}}(a_i-1).
\]

Now fix an arbitrary subset \(S\subseteq\{1,\dots,k\}\).
Every term \(\prod_{i\notin S}(a_i-1)\) appears exactly $|S|$ times in the double sum above --- once for each index \(j\in S\).
Grouping equal terms therefore yields
\[
F_2(a_1,\dots,a_k)=\sum_{S\subseteq\{1,\dots,k\}}
   |S|\cdot \prod_{i\notin S}(a_i-1)
   =D_2(a_1,\dots,a_k).
\]
Thus \(F_2=D_2\).
\end{proof}

\begin{corollary}\label{cor:equality:books}
    For a bipartite generalized book graph formed by $k$ cycles of lengths $c_1, \ldots, c_k$, Expressions \eqref{eq:F:book} and \eqref{eq:D:bipartite:book} are equal.
\end{corollary}
\begin{proof}
    Take $a_i = 2^{c_i - 1} - 1$ in Lemmas \ref{lem:F_1=D_1} and \ref{lem:F_2=D_2}.
\end{proof}

\begin{lemma}\label{lem:nonbipartite:book}
    If $B$ is a non-bipartite generalized book graph, then $|\F(B)| < |\D(B)|$.
\end{lemma}
\begin{proof}
     Assume that $B$ is formed by $k$ cycles of lengths $c_1, \ldots, c_k$, and at least one of the cycles is odd. We show that in this case $|\D(B)|$ is strictly bigger than Expression \eqref{eq:D:bipartite:book}. This is sufficient, since the latter is equal to $|\F(B)|$ by Proposition \ref{prop:F:book} and Corollary \ref{cor:equality:books}.

    Analogously to Lemma \ref{lem:deg:0,2}, consider a path $(l, v_1, \ldots, v_n, r)$, where $n$ is odd. Let $d \in \D(P)$. The only case, when we can not determine $d(l), d(r)$ by all $d(v_i)$ is when $d(v_i) = 1$ for all $i$. But in this case, since $n$ is odd, the two possibilities for $d(l)$ and $d(r)$ are $\begin{cases} d(l) = 1\\
    d(r) = 0\end{cases}$ and $\begin{cases} d(l) = 0 \\ d(r) = 1 \end{cases}$, see the illustration below:

    \medskip

    
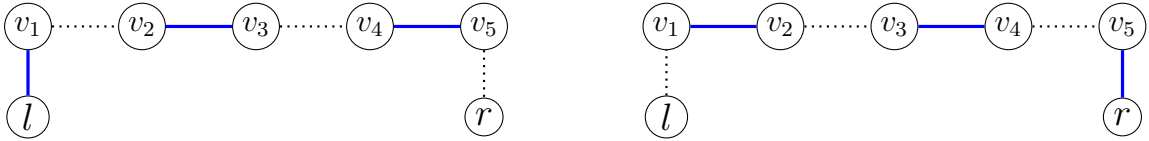
\begin{tikzpicture}[scale=0.6,
        every node/.style={circle,draw,fill=white,inner sep=1.8pt},
        edge/.style   ={very thick}]   


  \node (ll) at (0,0) {\large $l$};
  \node (l1) at (0,2) {$v_1$};
  \node (l2) at (2.5,2) {$v_2$};
  \node (l3) at (5,2) {$v_3$};
  \node (l4) at (7.5,2) {$v_4$};
  \node (l5) at (10,2) {$v_5$};
  \node (lr) at (10,0) {\large $r$};
  
  \node (rl) at (14,0) {\large $l$};
  \node (r1) at (14,2) {$v_1$};
  \node (r2) at (16.5,2) {$v_2$};
  \node (r3) at (19,2) {$v_3$};
  \node (r4) at (21.5,2) {$v_4$};
  \node (r5) at (24,2) {$v_5$};
  \node (rr) at (24,0) {\large $r$};

  \draw[blue, very thick] (ll)--(l1);
  \draw[blue, very thick] (l2)--(l3);
  \draw[blue, very thick] (l4)--(l5);
  
  \draw[blue, very thick] (r1)--(r2);
  \draw[blue, very thick] (r3)--(r4);
  \draw[blue, very thick] (r5)--(rr);

  \draw[dotted, thick] (rl)--(r1);
  \draw[dotted, thick] (r2)--(r3);
  \draw[dotted, thick] (r4)--(r5);
  
  \draw[dotted, thick] (l1)--(l2);
  \draw[dotted, thick] (l3)--(l4);
  \draw[dotted, thick] (l5)--(lr);

\end{tikzpicture}
    
    Thus the only part of the proof of Proposition \ref{prop:D:bipartite:book}, meaningfully affected by the presence of odd cycles is Step 4, where we compute the coefficients before $\sum\limits_{S \subseteq \{1, \ldots, k\}}  T_S$. For a bipartite graph these coefficients are always $(|S|+2)$. Let us derive the coefficients in the non-bipartite case.

    Denote $S_{odd} \coloneqq \{i \in S : \ c_i \text{ is odd}\}$ and $S_{even} \coloneqq \{i \in S : \ c_i \text{ is even}\}$. We can further split Expression \eqref{eq:D:l,r} by even and odd cycles. 
    \begin{equation*}
        \begin{cases}
        d(l) = \delta + \sum\limits_{i \in S_{even}} d_i(l) + \sum\limits_{h \in S_{odd}} d_h(l)+ \sum\limits_{j \notin S} d_j(l);\\
        d(r) = \delta + \sum\limits_{i \in S_{even}} d_i(r) + \sum\limits_{h \in S_{odd}} d_h(r)+ \sum\limits_{j \notin S} d_j(r).
        \end{cases}
    \end{equation*}
    Recall that $\sum\limits_{j \notin S} d_j(l)$ and $\sum\limits_{j \notin S} d_j(r)$ are determined, $\delta \in \{0, 1\}$, $d_i(l) = d_i(r) \in \{0, 1\}$ for $i \in S_{even}$, and $\begin{cases} d_h(l) = 1\\
    d_h(r) = 0\end{cases}$ or $\begin{cases} d_h(l) = 0 \\ d_h(r) = 1 \end{cases}$ for $h \in S_{odd}$.
    
    Denote $x\coloneqq \delta + \sum\limits_{i \in S_{even}} d_i(l) = \delta + \sum\limits_{i \in S_{even}} d_i(r)$ and $y \coloneqq \sum\limits_{h \in S_{odd}} d_h(l)$. Then $\sum\limits_{h \in S_{odd}} d_h(r) = |S_{odd}| - y$. Thus \[\begin{cases}
        d(l) = x+y + \sum\limits_{j \notin S} d_j(l);\\
        d(r) = x-y+|S_{odd}| + \sum\limits_{j \notin S} d_j(r),
        \end{cases} \text{ where $0 \leq x \leq |S_{even}| + 1$ and $0 \leq y \leq |S_{odd}|$.}\]
    It is easy to see that different pairs $x, y$ yield distinct pairs $d(l), d(r)$. Therefore, the required coefficient is $(|S_{even}| + 2)(|S_{odd}| + 1) \geq |S_{odd}| + |S_{even}| + 2 = |S|+2$. Moreover, the inequality is strict, whenever $|S_{odd}| > 0$.Finally, $|\D(B)|$ is strictly bigger than Expression \eqref{eq:D:bipartite:book}, which equals $|\F(B)|$.
\end{proof}

The following example illustrates Step 4 of the proof.
\begin{example}
Consider the case $|S_{odd}|=|S_{even}| = 1$. Then there are $(1+2)(1+1)=6$ options for the degrees of $l$ and $r$. On the picture below there are $6$ two-page books with all internal vertices of degree $1$. The degrees of $l, r$ can be $(0, 1), (1, 0), (1, 2), (2, 1), (2, 3), (3, 2)$.

\bigskip
\begin{tikzpicture}[scale=0.51,
    thick,
    blue,
    dotted/.style={dash pattern=on 1pt off 1pt},
    declare function={R=2;}
  ]
  \newcommand*\TriHex[7]{%
    \begin{scope}[xshift=#1cm]
      \path (0:R) coordinate (Rgt)
            (180:R) coordinate (Lft)
            (90:R) coordinate (Apex);
      \draw[#2] (Rgt) -- (Apex);   
      \draw[#3] (Lft) -- (Apex);   
      \draw[#4] (180:R) -- (240:R);
      \draw[#5] (240:R) -- (300:R);
      \draw[#6] (300:R) -- (360:R);
      \draw[#7] (0:R) -- (180:R);
    \end{scope}
  }

  \TriHex{ 0}{solid}{dotted}{dotted}{solid }{dotted}{dotted } 
  \TriHex{ 5}{dotted}{solid}{dotted}{solid }{dotted}{dotted } 
  \TriHex{10}{solid}{dotted}{dotted}{solid }{dotted}{solid } 
  \TriHex{15}{dotted}{solid}{dotted}{solid }{dotted}{solid } 
  \TriHex{20}{solid}{dotted}{solid}{dotted }{solid}{solid } 
  \TriHex{25}{dotted}{solid}{solid}{dotted }{solid}{solid } 
\end{tikzpicture}
\end{example}

\begin{remark}\label{rem:subbook}
    Hypothesis \ref{<hypo} holds for any subgraph of a generalized book graph. If a page in a subgraph is missing some ``internal'' edges, then we can discard the page using Remark \ref{rem:factorization:edge}. Thus the only subgraph we need to investigate is a book without the base. The proofs in this case are very similar to the generalized book case. Let $G$ be subgraph obtained by deleting the base from a book $B$ formed by $k$ cycles of lengths $c_1, \ldots, c_k$. 
    
    Denote $R \coloneqq \prod\limits_{i=1}^k (2^{c_i - 1} - 1)$. Then, by the proof to Proposition \ref{prop:F:book} we obtain $|\F(G)| = R + \sum\limits_{j=1}^{k}\frac{R}{2^{c_j - 1} - 1}.$ Now similarly to Proposition \ref{prop:D:bipartite:book} and Lemma \ref{lem:nonbipartite:book} we compute $|\D(G)|$. The only difference in Lemma \ref{lem:nonbipartite:book} is that due to the absence of the base, we obtain the coefficient $(|S_{even}| + 1)(|S_{odd}| + 1) \geq |S|+1$, where the inequality is strict, whenever $|S_{even}|, |S_{odd}| \neq 0$. Note that now (because there is no base) $G$ is bipartite if and only if $c_1, \ldots, c_k$ have the same parity. Therefore, we obtain that if $G$ is bipartite, then $|\D(B)| = \sum\limits_{S \subseteq \{1, \ldots, k\}} (|S| + 1) T_S$. The latter equals $|\F(G)|$ by Lemmas \ref{lem:F_1=D_1}, \ref{lem:F_2=D_2}. If $G$ is non-bipartite, then similarly to Lemma \ref{lem:nonbipartite:book}, we obtain $|\F(G)| < |\D(G)|$.

\end{remark}

\bibliographystyle{plain}
\bibliography{mybibfile}
	
\end{document}